\documentclass[11pt]{amsart}
\usepackage{amsmath,amsthm,amssymb,graphicx,mathtools,tikz,hyperref}
\usetikzlibrary{positioning}
\newcommand{\n}{\mathbb{N}}
\newcommand{\z}{\mathbb{Z}}

\newcommand{\cx}{\mathbb{C}}
\newcommand{\real}{\mathbb{R}}

\newtheorem{theorem}{Theorem}[section]
\newtheorem{proposition}[theorem]{Proposition}
\newtheorem{corollary}{Corollary}[theorem]
\newtheorem{lemma}[theorem]{Lemma}

\theoremstyle{definition}
\newtheorem*{definition}{Definition}
\hypersetup{
	colorlinks,
	linkcolor=blue
}
\usepackage{setspace}
\begin{document}
	\date{}
	
	\title{On the density of eigenvalues on periodic graphs}
	\author{Cosmas Kravaris} 
	
	\maketitle
	
\begin{abstract}
Suppose that $\Gamma=(V,E)$ is a graph with vertices $V$, edges $E$, a free group action on the vertices $\z^d \curvearrowright V$ with finitely many orbits, and a linear operator $D$ on the Hilbert space $l^2(V)$ such that $D$ commutes with the group action.
Fix $\lambda \in \mathbb{R}$ in the pure-point spectrum of $D$ and consider the vector space of all eigenfunctions of finite support $K$.
Then $K$ is a non-trivial finitely generated module over the ring of Laurent polynomials, and the density of $\lambda$ is given by an Euler-characteristic type formula by taking a finite free resolution of $K$.
Furthermore, these claims generalize under suitable assumptions to the non-commutative setting of a finite generated amenable group acting on the vertices freely with finitely many orbits, and commuting with the operator $D$.
\end{abstract}

\section{Introduction}

A major object of study in mathematical solid state physics is a crystal, a body of matter whose molecular structure is periodic.
Many discrete models of crystals consist of a graph $\Gamma=(V,E)$ with vertices $V$, edges $E$, a free group action on the vertices $\z^d \curvearrowright V$ with finitely many orbits, and a linear operator $D$ of finite order on the Hilbert space $l^2(V)$ such that $D$ commutes with the group action.

The following are two examples in their simplest forms.
The \textbf{discrete Schrödiger operator} acting on $l^2(\z^2)$ is given by
$$Df(m,n):=f(m+1,n) + f(m-1,n) + f(m,n+1) + f(m,n-1) + (q-4) f(m,n),$$
where $f\in l^2(\z^2)$, $m,n\in\z$, and $q\in \mathbb{R}$ is a uniform electric potential.
The \textbf{discrete magnetic Laplacian operator} or \textbf{Harper operator} \cite{ShubinMagnetic, SunadaMagnetic} acting on on $l^2(\z^2)$ is given by
$$Df(m,n):=e^{-i\pi\alpha n} f(m+1,n) + e^{i\pi\alpha n} f(m-1,n) + e^{i\pi\alpha m} f(m,n+1) + e^{-i\pi\alpha m} f(m,n-1),$$
where $f\in l^2(\z^2)$, $m,n\in\z$, and $\alpha \in \mathbb{R}$ is a uniform magnetic flux.
In both cases, the underlying graph $\Gamma$ is the square grid on the plane.
In the first case, the group action of $\z^2$ is given by unit 1 vertical and horizontal translations:
$$(x,y)\cdot(m,n):=(m+x,n+y)\;\;\;\;where\;\;x,y,m,n\in\z.$$
In the latter case, $D$ is periodic only when $\alpha$ is rational, say $\alpha=p/q$ where $p\in\z, q\in\mathbb{N}$.
Then the group action is given by unit $2q$ vertical and horizontal translations:
$$(x,y)\cdot(m,n):=(m + 2q x , n + 2q y)\;\;\;\;where\;\;x,y,m,n\in\z.$$
The spectrum of the above operators corresponds to the energy of the modeled crystal, and physical phenomena can be described by further spectral properties of $D$.

One example of such a spectral property is the \textbf{Bloch variety}, also known as the \textbf{dispersion relation}. 
It is an algebraic variety defined via the Floquet-Bloch transform $\widehat{D}$ of the operator $D$.
Recently, there has been a lot of work on the algebraic properties of the Bloch variety and their relations to spectral and physical properties.
For example, the question of whether the Bloch variety is generically the union of graphs of Morse functions is fundamental for the notion of effective mass in solid state physics \cite{DoKuchSott}.
Other algebraic properties of interest include the irreducibly of the Bloch variety \cite{Liu20, FLM} and toric compactifications \cite{DoKuchSott, FaustSottile}.
Furthermore, the book of Gieseker, Knoerrer and Trubowiz \cite{GKT} uses the algebraic geometry of the Bloch variety to study the density of states.

In this article, we study a new algebraic structure: the algebraic structure of \textbf{finite support eigenfunctions}.
Fix $\lambda \in \mathbb{R}$ in the pure-point spectrum of $D$, that is, the set of genuine eigenvalues.
Elements in the pure-point spectrum are in one to one correspondence to flat sheets in the Bloch variety.
By work of Kuchment \cite{Kuch}, there exists a non-zero eigenfunction $f$ of $\lambda$ (i.e. $D f = \lambda f$) which has finite support.
Therefore, the vector space of all eigenfunctions of finite support is nontrivial. We denote it by $K$.

The key insight of this article is that $K$ is a finitely generated $\cx[\z^d]$-module (where $\cx[\z^d] = \cx[z_1^{\pm 1},\dots,z_d^{\pm 1}]$ is the group algebra of $\z^d$, which is the same as the ring of Laurent polynomials).
Although finite support eigenfunctions have been studied before in \cite{Kuch, VeselicPercolationHamiltonians, Higuchi:Nomura:Paper, KagomeAndFiniteSupport, LenzVeselicUniformApproximationAndExamples}, the module structure was not used nor noted.
It can be described as follows.
Pick a free generating set $e_1, \dots, e_d$ of $\z^d$.
For each $e_i$ ($1\leq i \leq d$) and each $f \in K$ define
$$e_i \cdot f(v) := f(e_i^{-1} \cdot v) \;\;\;\; v \in V.$$
Via linearization, we get an action $\cx[\z^d] \curvearrowright K$ making $K$ a $\cx[\z^d]$-module.

The main result of this paper is an Euler-characteristic-type formula for the density of $\{\lambda\}$, denoted by $\nu(\{\lambda\})$.
The notion of the \textbf{density of states measure} originated in solid state physics \cite{KittelCondensedMatter} where it is, roughly speaking, the number of states (eigenfunctions) per unit volume whose energy (eigenvalue) lies in a given range.
For more on the density of states, see Section 2.
We write $R$ for the ring of Laurent polynomials $\cx[z_1^{\pm 1},\dots,z_d^{\pm 1}]$.
Here is the main result:

\begin{theorem}
	Let $K$ be the $R$-module of finite support eigenfunctions of $\lambda \in \mathbb{R}$ in the pure-point spectrum of $D$.
	Taking a finite resolution of $K$ by finite rank $R$-modules
	$$0 \rightarrow R^{r_d} \rightarrow R^{r_{d-1}} \;\dots\; \rightarrow R^{r_1} \rightarrow R^{r_0} \rightarrow K \rightarrow 0$$
	the density of states measure of $\{\lambda\}$ is given by the alternating sum of the ranks
	$$ \nu(\{\lambda\}) = \dfrac{1}{|V/\z^d|} \sum_{k=0}^{d} (-1)^k r_k $$
	where $|V/\z^d|$ is the number of orbits of the action $\z^d \curvearrowright V$.
\end{theorem}

The above resolution always exists due to the classical Hilbert-Syzygy theorem, and there exist well-known algorithms for the computation of $r_0, \dots, r_d$ \cite{UsingAlgGeom}.
Also, this result is very different from the results in \cite{GKT}, which use toroidal compactifications.

A lot of the results generalize to the non-commutative setting: that is, replacing the action $\z^d \curvearrowright V$ with an action on the vertices by a finitely generated group $G$.
In this case, the space of all finite support eigenfunctions is a module over the group algebra $\cx G$.
The notion of density of states also exists in the non-commutative setting (often under the names von Neumann trace and von Neumann dimension), and is intricately connected to the study of $l^2$-invariants in geometry and $K$-theory \cite{LuckLTwoInvariants}. 
For instance, a finitely generated torsion-free group $G$ satisfies the \textbf{strong Atiyah conjecture over $\cx$} if and only if for every $G$-periodic graph $\Gamma$, and linear periodic operator $D$ on $\Gamma$ of finite order, 
$|\Gamma/G|\cdot \nu(\{0\}) \in \z$.
The conjecture is known for large classes of groups; see Chapter 10 in \cite{LuckLTwoInvariants} for more details.
Theorem 1.1 implies the following well-known result:
\begin{corollary}
	The strong Atiyah conjecture over $\cx$ holds for free abelian groups.
\end{corollary}
Note that there are many other algebraic approaches related to the strong Atiyah conjecture, yielding much stronger results.
For example, in \cite{StrongAtiyahAndAlgebra} the conjecture is studied over the algebraic numbers $\bar{\mathbb{Q}}$.
Furthermore, the authors in \cite{StrongAtiyahAndAlgebra} prove an approximation result for the density of eigenvalues for direct and inverse limits of groups (which is completely different from Theorem 1.1 and from Lemma 5.1).

In group theory, the density of states of the Markov operator $M$ on the Cayley graph $\Gamma(G,S)$ of a finitely generated group $G$ is called the \textbf{spectral measure} of $G$ with respect to the generating set $S$.
In \cite{KestenClassic}, Kesten initiated the study of spectral measures of groups.
He showed that $G$ is amenable if and only if $1$ belongs to the spectrum of $M$ and computed the spectral measures of the free groups, all of which are absolutely continuous with respect to the Lebesgue measure.
The first example of a group with purely discrete spectral measure is the lamplighter group, as shown by Grigorchuk and Zuk in \cite{GriZukLamplighter}.
On the other hand, in \cite{GriPittetNoEigenv}, Grigorchuk and Pittet showed that for a different generating set, the lamplighter group has continuous spectral measure, so the pure-point spectrum is empty.
This indicates that the behavior of the spectral measure of a group depends on the generating set.

We begin with an introductory section on amenable $G$-periodic graphs, and recall the existence and density of finite support eigenfunctions.
Next, we focus on abelian $G$-periodic graphs, review the Floquet-Bloch transform, and show that for each eigenvalue, there are finitely many finite support eigenfunctions up to translations and linear combinations. 
We then generalize this claim to amenable $G$-periodic graphs with Noetherian group algebra and apply it to approximate the density of an eigenvalue using finite support eigenfunctions.
Next, we review syzygy modules and free resolutions and prove the main result (Theorem 1.1).
When $G$ has subexponential growth, we generalize the main result, provided that a finite free resolution exists.
Finally, we provide an example of a module of finite support eigenfunctions which is not free, and take a free resolution to compute the density of the corresponding eigenvalue via Theorem 1.1.
Some of the proofs of known results are included in Appendix A, and the computations for the example are included in Appendix B.

\section{Preliminaries on Amenable Periodic Graphs}

Let $\Gamma=(V,E)$ be a locally finite graph (so the degree of each vertex is finite) with set of vertices $V$ and set of edges $E$.
All graphs in this article will be simple and undirected.
Consider the Hilbert space of all complex valued square summable functions on $V$, denoted by $l^2(V)$.
We define the \textbf{discrete Laplacian} $\Delta: l^2(V) \rightarrow l^2(V)$ as
$$ \Delta f (v) := \dfrac{1}{\sqrt{\deg_{\Gamma}v}}\sum_{w\sim v} \; \Bigl(\dfrac{f(w)}{\sqrt{\deg_{\Gamma}w}}-\dfrac{f(v)}{\sqrt{\deg_{\Gamma}v}}\Bigr), $$ 
where $\deg_{\Gamma}v$ is the degree of the vertex $v$.
The operator $\Delta$ defines a self-adjoint operator on $l^2(V)$. 
We consider the \textbf{spectrum} $sp(\Delta)$ of $\Delta$, which is the set of all $\lambda\in \cx$ such that $(\Delta -\lambda I)$ does not have a bounded inverse. 
Since $\Delta$ is self-adjoint, $sp(\Delta) \subset \real$. 
The \textbf{pure-point spectrum} is the set of all $\lambda\in \cx$ for which $(\Delta -\lambda I)$ is not injective. 
An element $\lambda$ of the pure-point spectrum is called an \textbf{eigenvalue}. 
A function $f \in l^2(V)$ such that $(\Delta -\lambda I) f = 0$ is an \textbf{eigenfunction} of $\Delta$ corresponding to the eigenvalue $\lambda$ and the space of all such functions is the \textbf{eigenspace} of $\Delta$ corresponding to $\lambda$.

There are many variations on this study \cite{IDS_for_PDE_Survey, survey, AnalysisOnGraphsSurvey}. The most obvious one is to consider the \textbf{adjacency operator} $A$, \textbf{Markov operator} $M$ and \textbf{Schrödinger operator} $\Delta + q$ 
(where $q: V \rightarrow \mathbb{R}$ is bounded) defined as
$$ A  f (v) := \sum_{w\sim v} f (w) \;\;\;\;\;\;\;\; M  f (v) := \dfrac{1}{\deg_{\Gamma}v}\sum_{w\sim v} f(w) \;\;\;\;\;\;\;\; (\Delta + q)  f (v) := \Delta f(v) + q(v).$$
Another variation is to account for multiple edges, or to use edge weights.
Furthermore, one may also consider quantum graphs \cite{BerkKuchQuantum} in which edges are treated as unit intervals and analysis is performed on them.	

Now suppose that $G$ is a discrete group with finite generating set $S$.
We call $G$ \textbf{amenable} when there exists a sequence $\{F_j\}_{j=1}^{\infty}$ of finite subsets of $G$ such that 
$$\dfrac{|(F_j \cdot S) \setminus F_j|}{|F_j|} \to 0 \;\;as\;\; j \to \infty.$$ 
The sequence $\{F_j\}_j$ is called a \textbf{Følner sequence} for $(G,S)$.
Note that the amenability of $G$ does not depend on the finite generating set $S$ \cite{GGTDrutuKapovich}.

\begin{definition}
	Let G be a finitely generated group.
	A \textbf{$G$-periodic graph} is a graph $\Gamma=(V,E)$ which admits a free, cofinite and edge preserving action $G \curvearrowright V$. More precisely:
	\\ i. $G \curvearrowright V$ is a free action on the set of vertices $V$.
	\\ ii. The orbit space $V/G$ is finite.
	\\ iii. For all $g \in G, u, v, \in V$ 
	$$ (u,v) \in E \implies (g \cdot u,g \cdot v) \in E.$$
	Choosing one vertex from each orbit of the group action, we obtain a \textbf{fundamental domain} $W \subset V$ which is a finite subset (by ii).
	If $G$ is amenable, we call $\Gamma$ an amenable $G$-periodic graph.
	If $G$ is abelian, we call $\Gamma$ an abelian $G$-periodic graph.
\end{definition}

\textbf{For the rest of this section} $\Gamma = (V,E)$ will always denote a $G$-periodic graph.

\begin{definition}
	The \textbf{left-regular representation} of G associated to $\Gamma$ is the map $\pi: G \rightarrow \mathcal{U}(l^2(V))$ defined by
	$$ \pi_g f(v) := f(g^{-1} \cdot v) \;\;\; v \in V, f \in l^2(V), g \in G.$$
	It is a unitary representation of $G$ into the space of bounded unitary operators on $l^2(V)$.
	An operator $T: l^2(V) \rightarrow l^2(V)$ is called \textbf{periodic} whenever it commutes with the left-regular representation, i.e. $\pi_g T f = T \pi_g f$ for any $f\in l^2(V),\;g\in G$.
	Since $G \curvearrowright V$ preserves edges, it follows that the discrete Laplacian is periodic.
\end{definition}

\begin{definition}
	For any vertices $u,v \in V$ denote by $d(v,w)$ the length of the shortest path in $\Gamma$ from $v$ to $w$,
	taking value $\infty$ when no such path exists.
	The \textbf{r-thick boundary} of a subset $F \subset V$ (where $r \in \n$) is defined to be:
	$$\partial_r F := \{v \in V \backslash F : \;there\;exists\;\; w \in F \;\;with\;\; d(v,w) \leq r\}.$$
\end{definition}

Note that $\{F_j\}_j$ is a Følner sequence for $(G,S)$ when $|\partial_1 F_j|/|F_j| \to 0$ on the Cayley graph $\Gamma(G,S)$ (its set of vertices is $G$ and set of edges is $\{(g, g \cdot s):g \in G\;s\in S\}$).
Also note that when $f \in l^2(V)$ is an eigenfunction and there exists $F \subset G$ with $f \equiv 0$ on $\partial_2 F$, then $I_F f$ is also an eigenfunction.
Here, $I_F$ is the projection operator onto functions supported on $F$, i.e. $I_F f (v) = f (v)$ when $v \in F$ and is $0$ otherwise.
In general, an operator $T: l^2(V) \rightarrow l^2(V)$ is said to be of \textbf{finite order r} if for any $f, g \in l^2(V), v \in V$,
$$f(w)=g(w) \;\;for\;all\;\; w \;\;with\;\; d(w,v) < r \implies Tf(v) = Tg(v).$$
The operator $T$ then has the following property: when $f \in l^2(V)$ is an eigenfunction associated to $T$ and there exists $F \subset V$ with $f \equiv 0$ on $\partial_r F$, then $I_F f$ is also an eigenfunction.
The discrete Laplacian, the Adjacency, the Markov and the Schrödinger operators are all of order 2.
Although we focus in this paper on the discrete Laplacian, all of the claims and techniques hold for periodic linear operators of finite order on a periodic graph (with minor changes in the constants related to the order of the operator).

The following technical lemma is well-known. 
For completeness' sake, its proof is included in Appendix A.

\begin{lemma}[Thick Følner sequences]
	If $G$ is amenable,
	then for any thickness $r \in \n$, 
	generating set $S$ of $G$
	and fundamental domain $W$ of $\Gamma$
	there exists $l \in \n$ and a sequence of finite subsets 
	$\mathcal{F}_j \subset G$ 
	such that the sequence 
	$F_j := \mathcal{F}_j \cdot W \subset V$ 
	satisfies
	$$\partial_r F_j \subset (\partial_l \mathcal{F}_j)\cdot W\;\;for\;all\;j,\;\;\;and\;\;\;\dfrac{|\partial_r F_j|}{|F_j|} \leq \dfrac{|(\partial_l \mathcal{F}_j)\cdot W|}{|F_j|} \to 0 \;as\;j\to\infty.$$
	We call the sequence $\{F_j\}$ a \textbf{standard r-thick Følner sequence} 
	with respect to a fixed fundamental domain $W$ and generating set $S$ of $G$.
\end{lemma}

We next discuss the concept of density.
According to the spectral theorem of self-adjoint operators (see \cite{ReedSimonClassic}), from $\Delta$ we obtain a spectral measure $E$ whose input are Borel sets and outputs are projections on $l^2(V)$.
In the case $B=\{\lambda\}$ where $\lambda$ is an eigenvalue, $E(\{\lambda\})$ is the orthonormal projection onto the eigenspace of $\lambda$.
We will denote this eigenspace by $E_{\lambda}$.

\begin{definition}
	Fix a fundamental domain $W$ of $\Gamma$.
	The \textbf{density} or \textbf{density of states measure} or \textbf{von Neumann trace} of a Borel subset $B \subset sp(\Delta)$ is
	$$ \nu(B) := \dfrac{1}{|W|} tr(E(B)I_W),$$
	where $tr(\cdot)$ is the usual trace of a Hilbert space operator and $I_W$ is the standard projection $l^2(V) \twoheadrightarrow l^2(W)$ (which is a finite rank operator hence $E(B)I_W$ is of trace class).
\end{definition}

Note that we may commute the operators inside the trace: $tr(E(B)I_W) = tr(I_W E(B)) = tr(I_W E(B) I_W)$ (see \cite{ReedSimonClassic}).
From the spectral theorem, it follows that $\nu(\cdot)$ is a measure on $sp(\Delta)$.
It is well known (for instance see \cite{GriPittetNoEigenv}) that this measure is purely continuous except a set of point masses which occur precisely at the point spectrum of $\Delta$ (i.e. the set of eigenvalues).
When $\lambda$ is an eigenvalue, $\nu(\{\lambda\})$ is called the \textbf{von Neumann dimension} of the eigenspace $E_\lambda$ (see \cite{LuckLTwoInvariants, GriPittetNoEigenv} for further context) 
and when $\Gamma = \Gamma(G,S)$ is a Cayley graph, $\nu(B) = \langle E(B) \delta_1, \delta_1\rangle$ and is often called the \textbf{spectral measure} of $(G,S)$.

The notation $\nu(\cdot)$ for the density of states measure is mostly standard, although the notation $dk(\cdot)$ is also used sometimes (e.g. see \cite{Higuchi:Nomura:Paper,GriPittetNoEigenv}).
Another standard notation is that of the cumulative distribution function of $\nu(\cdot)$, called the \textbf{integrated density of states}:
$N(\lambda):=\nu((-\infty, \lambda)), \lambda\in\mathbb{R}$.

Denote by $D(\Gamma)$ all $\mathbb{C}$-valued functions on the vertices $V$ of $\Gamma$ with finite support
and by $D_{\lambda}(\Gamma)$ all the eigenfunctions of $\lambda$ in $D(\Gamma)$.
The following two theorems are due to Kuchment \cite{Kuch} for the case when $G$ is abelian.
In \cite{Higuchi:Nomura:Paper}, Veseli\'{c} generalized the two theorems to the case when $G$ is amenable.
Their proofs are included in Appendix A: the proof of Theorem 2.6 is due to Higuchi and Nomura \cite{Higuchi:Nomura:Paper}, and the proof of Theorem 2.7 is a modification of this proof.
All the proofs in the amenable case use an argument of Delyon and Souillard  \cite{DelyonSouillardOriginalIdea}.

\begin{theorem}[Strong Localization of Eigenfunctions, Kuchment-Veseli\'{c}]
	Let $\Gamma$ be a $G$-periodic graph with amenable group $G$ and let $\Delta$ be the Laplacian operator on it. 
	If $\lambda$ is an eigenvalue of $\Delta$, then there exists an eigenfunction of $\lambda$ which has finite support, i.e. $D_{\lambda}(\Gamma) \neq \{0\}$.
\end{theorem}

\begin{theorem}[Finite Support Approximation of Eigenfunctions, Kuchment-Veseli\'{c}]
	Let $\Gamma$ be an $G$-periodic graph with amenable $G$ and let $\Delta$ be the Laplacian operator on it with eigenvalue $\lambda$. 
	If $f \in l^2(V)$ is an eigenfunction of $\lambda$,
	then for all $\epsilon>0$ arbitrarily small, there exists $g \in D_{\lambda}(\Gamma)$ such that $||f-g||<\epsilon$,
	i.e. the finite support eigenfunctions of $\lambda$ are $l^2$-dense in the $l^2$-eigenspace of $\lambda$. 
\end{theorem}

\section{Abelian Periodic Graphs}

\textbf{For this section}, let $\Gamma = (V,E)$ be a $\z^d$-periodic graph with fundamental domain $W$.

The \textbf{Floquet-Bloch transform} of $f\in l^2(V)$ is a complex valued function with domain $V \times \mathbb{T}^d $ (where $\mathbb{T}^d$ is the $d$-dimensional torus)
$$\widehat{f}(v,e^{i k}) := \sum_{g \in \z ^d} f(g \cdot v) e^{-i k \cdot g},$$
where $v\in V$, $k \in \mathbb{R}^d / 2\pi \z^d$, $e^{i k}=(e^{i k_j})_{j=1}^d \in \mathbb{T}^d$ and $k \cdot g$ is the standard dot product. 

One can verify that for all $g \in \z^d, \widehat{f}(g \cdot v,e^{i k}) = e^{i g \cdot k} \widehat{f}(v,e^{i k})$. 
This means that the entire function $\widehat{f}$ may be recovered from its restriction to $W \times \mathbb{T}^d$, hence from now on we will view $\widehat{f}$ as a function on this restricted domain. 
The following key theorem is a consequence of standard techniques from Fourier analysis. For a more detailed exposition, see Chapter 4 \cite{BerkKuchQuantum}.

\begin{proposition}
	a) Inversion Formula: 
	for all $v \in W, g \in \z^d$,
	$$ f(g \cdot v) = \int_{[-\pi,\pi]^d} \widehat{f}(v,e^{i k})e^{i k \cdot g}dk$$
	b) The map 
	$$f \mapsto (2\pi)^{-n/2} (\widehat{f}(v,e^{i k}))_{v \in W}$$
	is a unitary map from $l^2(V)$ to $L^2(\mathbb{T}^d, \cx^{|W|})$, the space of all square summable functions from $\mathbb{T}^d$ to $\cx^{|W|}$.
\end{proposition}

As a result of the above theorem, by composing $\Delta$ with the Floquet-Bloch transform and its inverse, we may transform the discrete Laplacian $\Delta:l^2(\Gamma)\rightarrow l^2(\Gamma)$ to a corresponding self-adjoint map $\widehat{\Delta}: L^2(\mathbb{T}^d,\cx^{|W|}) \rightarrow L^2(\mathbb{T}^d, \cx^{|W|})$.

Let $v_1:=(\delta_{1,i})_{i=1}^d, \dots v_d:=(\delta_{d,i})_{i=1}^d$ be the standard basis of $\z^d$ and write $z_j := e^{i v_j}$.
That way, $e^{i k \cdot g} = z^g$.
By definition, the image of $f \in D(\Gamma)$ under the Floquet-Bloch transform is a vector of size $n=|W|$ whose entries are Laurent polynomials in $(z_j)_j$ (that is, polynomials in $\{z_j\}_j \cup \{z_j^{-1}\}_j$). 
We denote this ring of Laurent polynomials by $\cx[z_1^{\pm 1},z_2^{\pm 1}, \dots ,z_d^{\pm 1}]$, and the vectors of size $n$  with Laurent polynomial entries by $\bigoplus_{k=1}^n \cx[z_1^{\pm 1},\dots,z_d^{\pm 1}]$.
Notice that $\bigoplus_{k=1}^n \cx[z_1^{\pm 1},\dots,z_d^{\pm 1}] \subset L^2(\mathbb{T}^d,\cx^{|W|})$, and since $f \in D(\Gamma) \implies \Delta f \in D(\Gamma)$, we also have:
$$\widehat{\Delta} ( \bigoplus_{k=1}^n \cx[z_1^{\pm 1},\dots,z_d^{\pm 1}] ) \subset \bigoplus_{k=1}^n \cx[z_1^{\pm 1},\dots,z_d^{\pm 1}].$$
The next two propositions allow us to pass questions about finite support eigenfunctions to questions in commutative algebra.

\begin{proposition}
	The map 
	$$\widehat{\Delta}: \bigoplus_{k=1}^n \cx[z_1^{\pm 1},\dots,z_d^{\pm 1}] \to \bigoplus_{k=1}^n \cx[z_1^{\pm 1},\dots,z_d^{\pm 1}]$$
	is a $\cx[z_1^{\pm 1},\dots,z_d^{\pm 1}]$-module homomorphism.
\end{proposition}
\begin{proof}
	To see that $\widehat{\Delta}$ respects multiplication by monomials $z^g$ (where $g\in\z^d$) notice that $\pi_g \Delta = \Delta \pi_g$.
	Under the Floquet-Bloch transform, this equation becomes $z^g \widehat{\Delta} = \widehat{\Delta} z^g$.
	Finally, the linearity of $\widehat{\Delta}$ means that $\widehat{\Delta}$ also respects linear combinations of monomials.
	To be more precise, let $\sum_{g \in \z^d} a_g z^g \in \cx[z_1^{\pm 1},\dots,z_d^{\pm 1}]$. For all $\widehat{f}\in \bigoplus_{k=1}^n \cx[z_1^{\pm 1},\dots,z_d^{\pm 1}]$ we have:
	$$\widehat{\Delta} (\sum_{g \in \z^d} a_g z^g ) \widehat{f} 
	= \widehat{\Delta} \sum_{g \in \z^d} a_g (z^g \widehat{f})
	= \sum_{g \in \z^d} a_g \widehat{\Delta}(z^g \widehat{f})
	= \sum_{g \in \z^d} a_g z^g \widehat{\Delta} \widehat{f},$$
	therefore we get a $\cx[z_1^{\pm 1},\dots,z_d^{\pm 1}]$-module homomorphism.
\end{proof}

A well-known consequence of the above proposition is that we may express $\widehat{\Delta}$ as a $|W|\times |W|$ matrix whose entries are rational functions on $\mathbb{T}^d \subset \cx^d$ expressed via Laurent polynomials.
Classical Floquet-Bloch theory shows that $\lambda$ is an eigenvalue if and only if $det(\widehat{\Delta}-\lambda I)$ is the zero function on $\mathbb{T}^d $.
Moreover, $\lambda$ lies in the spectrum $sp(\Delta)$ if and only if
$det(\widehat{\Delta}-\lambda I)$ has a zero.
For more details see Chapter 4 in \cite{BerkKuchQuantum}.
The following proposition is due to Kuchment \cite{Kuch}.

\begin{proposition}[Kuchment]
	A $\z ^d$-periodic graph $\Gamma$ with Laplacian $\Delta$ has an eigenvalue $\lambda$ if and only if the map
	$$(\widehat{\Delta} - \lambda I_n) : \bigoplus_{k=1}^n \cx[z_1^{\pm 1},\dotsc,z_d^{\pm 1}] \rightarrow \bigoplus_{k=1}^n \cx[z_1^{\pm 1},\dotsc,z_d^{\pm 1}]$$ is \textbf{not} injective.
	The kernel of this map corresponds to the finite support eigenfunctions of $\lambda$. 
\end{proposition}
\begin{proof}
	By the Floquet-Bloch transform, for all $f\in D(\Gamma)$: $(\widehat{\Delta}-\lambda I_n)\widehat{f} = 0 \iff (\Delta - \lambda I) f = 0$.
	By Theorem 2.6, $(\Delta - \lambda I) f = 0$ has a solution in $D(\Gamma)$ if and only if $\lambda$ is an eigenvalue of $\Delta$.
	Combining these two observations, the proposition follows.
\end{proof}
In view of the formula 
$$\widehat{f}(g \cdot v,e^{i k}) = e^{i g \cdot k} \widehat{f}(v,e^{i k}) = z_1^{g_1} \dotsc z_d^{g_d} \widehat{f}(v,e^{i k}) = z^g \widehat{f}(v,e^{i k}),$$
we see that when we multiply each component of $\widehat{f}$ by the same monomial $z^g \in \cx[z_1^{\pm 1},\dotsc,z_d^{\pm 1}]$ we are essentially translating $f$ by $g \in \z^d$. 
It follows that when we multiply each component of $\widehat{f}$ by an arbitrary element of $\cx[z_1^{\pm 1},\dotsc,z_d^{\pm 1}]$, then we are taking linear combinations of translations of $\widehat{f}$. 
Notice that $(\widehat{\Delta} - \lambda I_n)\widehat{f} = 0 \iff (\widehat{\Delta} - \lambda I_n) z^g \widehat{f} = 0$ hence translations of eigenfunctions are still eigenfunctions with respect to the same eigenvalue.
Therefore, if we wish to describe all finite support eigenfunctions of $\lambda$, it suffices to find  them up to translations by $\z^d$.

The following proposition is due to Kuchment \cite{Kuch}, but the proof presented here is new:

\begin{proposition}[Kuchment]
	Let $\Gamma$ be a $\z^d$-periodic graph with discrete Laplacian $\Delta$ on it and let $\lambda$ be an eigenvalue of $\Delta$.
	Then $\lambda$ has finitely many finite support eigenfunctions up to translation and linear combinations. That is, there are finite support eigenfunctions of $\lambda$, $f^{(1)}, \dotsc f^{(r)}$ such that every eigenfunction of $\lambda$, $f$ with finite support is the finite linear combination of translations of $f^{(1)}, \dotsc f^{(r)}$.
\end{proposition}

\begin{proof}
	Suppose that we had, a priori, a finite support eigenfunction $f \in D_{\lambda}(\Gamma)$. Then we may translate it such that, without loss of generality, $f$ has support in $\cup_{g \in \n^d} \;\text{-} g \cdot W$. This means that $\widehat{f}\in \cx[z_1,\dotsc,z_d]$. Next consider the entries of $(\widehat{\Delta} - \lambda I_n)$ (which are elements of the ring $\cx[z_1^{\pm 1},\dotsc,z_d^{\pm 1}]$), look at all the integer powers of $z_1,\dotsc,z_d$ in the terms of the entries and pick the smallest negative power $-P$ (set $P=0$ if all the powers are non-negative). That way, the entries of $(z_1,\dotsc,z_d)^P(\widehat{\Delta} - \lambda I_n)$ all lie in $\cx[z_1,\dotsc,z_d]$. We conclude that the set of all eigenfunctions of $\lambda$ whose support is finite and lies in $\cup_{g \in \n^d} \;\text{-} g \cdot W$ is the kernel of the $\cx[z_1,\dotsc,z_d]$-linear map:
	$$(z_1,\dotsc,z_d)^P(\widehat{\Delta} - \lambda I_n): \bigoplus_{k=1}^n \cx[z_1,\dotsc,z_d] \rightarrow \bigoplus_{k=1}^n \cx[z_1,\dotsc,z_d].$$
	
	By the classical Hilbert Basis Theorem, every ideal of $\cx[z_1,\dotsc,z_d]$ is finitely generated, i.e. $\cx[z_1,\dotsc,z_d]$ is Noetherian. 
	Every finitely generated module over a Noetherian ring is a Noetherian module. 
	The kernel of $(z_1,\dotsc,z_d)^P(\widehat{\Delta} - \lambda I_n)$ is certainly a submodule of the finitely generated $\cx[z_1,\dotsc,z_d]$-module $\bigoplus_{k=1}^n \cx[z_1,\dotsc,z_d]$, so it is finitely generated, say by generators $f^{(1)}, \dotsc f^{(r)}$. But what does this mean? For every eigenfunction with Floquet-Bloch transform $f$ there are $g_1, g_2, \dotsc g_r \in \cx[z_1,\dotsc,z_d]$ such that $f = g_1 f^{(1)} +  g_2 f^{(2)} + \dotsc + g_r f^{(r)}$. Breaking down $g_1, g_2, \dotsc, g_r$ into linear combinations of monomials, and noting that multiplication by $z^g$ corresponds to translation in $\Gamma$ by $g$, we see that $f$ is the linear combination of translations of $f^{(1)}, \dotsc, f^{(r)}$, and the claim follows.
\end{proof}

\section{Amenable Periodic Graphs with Noetherian Group Algebra}

In this section we generalize Proposition 3.4.
Recall that the Hilbert basis theorem (polynomial rings are Noetherian) was the key ingredient in proving Proposition 3.4.\\
\textbf{Throughout this section} $\Gamma$ is a $G$-periodic graph, where $G$ is amenable.

\begin{definition}[Noncommutative Floquet-Bloch transform]
	Associate to each $f \in D(\Gamma)$ a function $\widehat{f}:V \to \cx[G]$
	sending $v \in V$ to $$\sum_{g \in G} f(g^{-1} \cdot v) g \in \cx[G].$$
	Since $f$ has finite support, the sum is finite and the map is well defined.
	Notice that 
	$$for\;all\;\; h \in G, \;\; \widehat{(\pi_h f)}(v) = \widehat{f}(h^{-1} \cdot v) = h^{-1} \widehat{f}(v).$$
	Fix a fundamental domain $W$.
	In view of the above identity, we can recover $\widehat{f}:V \to \cx[G]$ from $\widehat{f}|_W$, i.e. the vector
	$(\widehat{f}(w))_{w\in W}  = (\sum_{g \in G} f(g \cdot w) g)_{w\in W} \in \bigoplus_{k=1}^n \cx[G]$, where $n:=|W|$.
\end{definition}

It is easy to see that $\widehat{\cdot}: D(\Gamma) \to \bigoplus_{k=1}^n \cx[G]$ is a bijective $\cx$-linear map.
To the operator $\Delta: D(\Gamma) \to D(\Gamma)$ corresponds some other operator
$\widehat{\Delta}: \bigoplus_{k=1}^n \cx[G] \to \bigoplus_{k=1}^n \cx[G]$.

\begin{proposition}
	The map 
	$$\widehat{\Delta}: \bigoplus_{k=1}^n \cx[G] \to \bigoplus_{k=1}^n \cx[G]$$
	is a left $\cx[G]$-module homomorphism.
\end{proposition}
\begin{proof}
	To see that $\widehat{\Delta}$ respects left multiplication by terms $g \in \cx[G]$ (where $g\in G$) notice that $\pi_g \Delta = \Delta \pi_g$.
	Under the Noncommutative Floquet-Bloch transform, this equation becomes $g \widehat{\Delta} = \widehat{\Delta} g$.
	For all $\sum_{g \in G} a_g g \in \cx[G]$ and $\widehat{f}\in \bigoplus_{k=1}^n \cx[G]$ we have:
	$$\widehat{\Delta} (\sum_{g \in G} a_g g ) \widehat{f} 
	= \widehat{\Delta} \sum_{g \in G} a_g (g \widehat{f})
	= \sum_{g \in G} a_g \widehat{\Delta}(g \widehat{f})
	= \sum_{g \in G} a_g g \widehat{\Delta} \widehat{f},$$
	therefore we get a $\cx[G]$-module homomorphism.
\end{proof}

Recall that a ring $R$ is \textbf{Noetherian} whenever every submodule of a finitely generated $R$-module is finitely generated.

\begin{proposition}
	Let $\Gamma$ be an $G$-periodic graph with amenable group $G$, $\Delta$ be the discrete Laplacian on $\Gamma$ and $\lambda$ an eigenvalue of $\Delta$.\\
	\textbf{If} the group algebra $\cx[G]$ is Noetherian (in particular if G is virtually polycyclic),\\
	\textbf{then} there are finitely many finite support eigenfunctions $f^{(1)},\dotsc, f^{(r)} \in D_\lambda(\Gamma)$
	such that every eigenfunction of $\lambda$ with finite support on $\Gamma$ is the finite linear combination
	of translations of $f^{(1)},\dotsc, f^{(r)}$
\end{proposition}

\begin{proof}
	Note that $f \in D_{\lambda}(\Gamma)$ if and only if $(\Delta - \lambda I)f = 0$
	if and only if $(\widehat{\Delta} - \lambda I) \widehat{f} = 0$ if and only if $f \in kernel(\widehat{\Delta} - \lambda I)$.
	This kernel is a submodule of $\bigoplus_{k=1}^n \cx[G]$.
	Since $\cx [G]$ is Noetherian,
	$kernel(\widehat{\Delta} - \lambda I)$ is finitely generated.
	Say, $kernel(\widehat{\Delta} - \lambda I) = \langle \widehat{f}^{(1)},\dotsc, \widehat{f}^{(r)}\rangle$.
	Then for all $f \in D_{\lambda}(\Gamma)$, there exist $h_1, \dotsc, h_r \in \cx[G]$ such that
	$ \widehat{f} = h_1 \widehat{f}^{(1)} + \dotsc + h_r \widehat{f}^{(r)}$. 
	Write $h_1 = \sum_{g \in G} a_g g$ so that: 
	$$h_1 \widehat{f}^{(1)} = (\sum_{g \in G} a_g g) \widehat{f}^{(1)} 
	= \sum_{g \in G} a_g (g \widehat{f}^{(1)}) = \sum_{g \in G} a_g (\widehat{\pi_g^{-1} f^{(1)}}).$$
	Hence to $h_1 \widehat{f}^{(1)}$ corresponds a function which is the finite linear combination of translations of $f^{(1)}$, and the theorem follows. 
\end{proof}

The theorem of Hall \cite{PHallPolycyclic} states that the group algebra of a finite extension of a polycyclic group is Noetherian (note that polycyclic groups are solvable and amenable, and they include all nilpotent groups \cite{GGTDrutuKapovich}).

\section{Finite Support Approximation of the Density of an Eigenvalue}

The following lemma allows us to study the density of an eigenvalue via its finite support eigenfunctions. 
Essentially, it connects the definition of density of states via a trace formula to the intuitive definition of density of states as the number of eigenfunctions per unit volume. 
The formula goes back to the work of Pastur \cite{Pastur, PasturFigotin} and Shubin \cite{Shubin, ShubinMagnetic}.

\begin{lemma}
	Let $\Gamma$ be a $G$-periodic graph with amenable $G$, discrete Laplacian $\Delta$ and let $\lambda$ be an eigenvalue of $\Delta$.
	\\\textbf{If} $\widehat{D_{\lambda}(\Gamma)} := kernel(\widehat{\Delta}-\lambda I)$ is a finitely generated $\cx G$-module,
	\\\textbf{then} for any fundamental domain $W$ and any generating set $S$ of $G$ there exists $j_0 \in \n$ such that 
	any standard $j_0$-thick Følner sequence $\{F_j\}$ of $\Gamma$ satisfies the following formula:
	$$ \nu(\{\lambda\}) = lim_{j \to \infty} \dfrac{dim_{\cx}\{f \in D_{\lambda}(\Gamma) : supp(f) \subset F_j\}}{|F_j|}.$$
\end{lemma}

The lemma follows from earlier work of Lenz and Veseli\'{c}, in particular Theorem 2.4 in \cite{LenzVeselicUniformApproximationAndExamples} (by estimating $lim_{\epsilon\to 0^{+}} \nu((\lambda-\epsilon,\lambda+\epsilon])$, and interchanging the limits thanks to uniform convergence).
Furthermore, the assumption that $\widehat{D_{\lambda}(\Gamma)}$ is a finitely generated $\cx G$-module is not necessary.
However, this same assumption will hold for the main results in Section 6.
Moreover, by Proposition 4.3 and Hall's theorem \cite{PHallPolycyclic}, this assumption holds for all finite extensions of polycyclic groups.
Finally, under this assumption, the proof of Lemma 5.1 is simple, and is included in Appendix A.

\section{Free Resolution Formula for the Density of an Eigenvalue}

Whenever we have a $G$-periodic graph $\Gamma$ with $G$ finitely generated amenable, we know that the $\cx G$-module of finite support eigenfunctions $K:=kernel(\widehat{\Delta} - \lambda I)$ is nonempty and it's dense in the $l^2$-eigenspace.
The goal of this section is to use the algebraic structure of $K$ to find the density of $\lambda$.
We would like to apply Lemma 5.1 and use $K$ to estimate $dim_{\cx}\{f \in D_{\lambda}(\Gamma) : supp(f) \subset F_j\}$.
The obvious way is to pick a generating set $f^{(1)}, f^{(2)}, \dots, f^{(r)}$ of $K$ and count all $g \in G$ and $1 \leq i \leq r$ such that $supp(\pi_g f^{(i)})\subset F_j$.
The issue is that the set of all those $\pi_g f^{(i)}$ may not be linearly independent, and hence our estimate can be far from optimal.
This motivates us to consider syzygy modules.

Let $R$ be a Noetherian ring.
If $\{f^{(1)}, f^{(2)},\dots, f^{(r)}\}$ is a finite generating set of a finitely generated $R$-module $M$, the \textbf{(first) syzygy module} of of $M$ with respect to the generators $\{f^{(1)}, f^{(2)},\dots f^{(r)}\}$ is the set of all $g = (g_1, \dots, g_r) \in R^r$ such that
$$ g_1 f^{(1)} +  g_2 f^{(2)} + \dots + g_r f^{(r)} = 0,$$
and is denoted by $Syz(M)$. 
Via pointwise multiplication $Syz(M)$ is an $R$-module as well.
Since $R$ is Noetherian, $Syz(M) \subset R^r$ is finitely generated.
Picking a finite set of generators for $Syz(M)$, we can consider its own syzygy module, $Syz(Syz(M))$, abbreviated by $Syz^2(M)$.
By iteration we can define the \textbf{(higher) syzygy module} $Syz^k(M)$ (which is finitely generated) for any positive power $k$ along with the conventions  $Syz^1(M) = Syz(M)$ and $Syz^0(M) = M$.

A alternative way to describe syzygies is through free resolutions. 
Picking a finite generating set $\{f^{(1)}, f^{(2)},\dots, f^{(r_0)}\}$ for $M$ is equivalent to finding a surjection $R^{r_0} \rightarrow M \rightarrow 0$. The kernel of this map is precisely the first syzygy module, so we get the Short Exact Sequence $0 \rightarrow Syz(M) \rightarrow R^{r_0} \rightarrow M \rightarrow 0$. Iterating this proccess we get another Short Exact Sequence $0 \rightarrow Syz^2(M) \rightarrow R^{r_1} \rightarrow Syz(M) \rightarrow 0$ where we choose a generating set of $Syz(M)$ of length $r_1$. We end up with the following sequence of maps:
$$\dots\; 
\twoheadrightarrow Syz^r(M) \hookrightarrow R^{r_3} 
\twoheadrightarrow Syz^3(M) \hookrightarrow R^{r_2} 
\twoheadrightarrow Syz^2(M) \hookrightarrow R^{r_1} 
\twoheadrightarrow Syz(M)   \hookrightarrow R^{r_0} 
\twoheadrightarrow M \rightarrow 0.$$
Via composition we get a \textbf{free resolution} of $M$, that is, a Long Exact Sequence beginning with $\rightarrow R^{r_0} \rightarrow M \rightarrow 0$ and consequently consisting of free R-modules:
$$\dots\;  \rightarrow R^{r_3}  \rightarrow R^{r_2} \rightarrow R^{r_1} \rightarrow R^{r_0} \rightarrow M \rightarrow 0.$$	
Then $Syz^k(M)$ can be recovered as the kernels (or equivalently images) of each map.
Note that when $R$ is \textbf{not} Noetherian, we can still construct syzygy modules, however, we cannot guarrantee that the free modules in the resulting resolution will be finitely generated.

Hilbert's Syzygy Theorem (see \cite{UsingAlgGeom}) asserts that the $d^{th}$ syzygy module $Syz^d(M)$ of a finitely generated $\cx[z_1^{\pm 1},\dots,z_d^{\pm 1}]$-module $M$ is always free. 
This means that if we choose a free generating set for $Syz^d(M)$, then $Syz^{d+1}(M)=0$. As a result the correspoding free resolution will terminate at the $(d+1)^{th}$ step.\\
\textbf{Until we state otherwise, $R = \cx[z_1^{\pm 1},\dotsc,z_d^{\pm 1}]$.}

\begin{theorem}
	Suppose that $\Gamma$ is a $\z^d$-periodic graph with fundamental domain $W$ and $\lambda$ is an eigenvalue of the Laplacian $\Delta$ on $\Gamma$.
	Let $K$ be the kernel of the $R$-module homomorphism
	$$\widehat{\Delta} - \lambda I :  \bigoplus_{k=1}^{W} R \to \bigoplus_{k=1}^{W} R.$$
	Then the following formula about the density of $\{\lambda\}$ holds:
	$$ \nu(\{\lambda\}) = \dfrac{1}{|W|} \sum_{k=0}^{d} (-1)^k r_k,$$
	where $r_0, \dotsc, r_d$ are the ranks of the free modules in a free resolution of $K$
	$$0 \rightarrow R^{r_d} \rightarrow R^{r_{d-1}} \;\dotsc\; \rightarrow R^{r_1} \rightarrow R^{r_0} \rightarrow K \rightarrow 0.$$
\end{theorem}

We remark that there exist well-known algorithms for the computation of $K$ and its higher syzygy modules, as well as software for these algorithms (see \cite{UsingAlgGeom}).
We will use the following notation for an $R$-submodule $M$ of a free $R$-module $R^n$ ($n \in \n$).
For each monomial $z^l = z_1^{l_1},\dotsc, z_d^{l_d}$ let $|z^l| := |l| = max_{1 \leq i \leq d} |l_i|$.
Next, for all $f_1 \in R$, let $|f_1|$ be the maximum length of the monomials it is comprised of 
and for all $f = (f_1, \dotsc, f_n) \in R^n$ let $|f| := max_{1 \leq k \leq n} |f_k|$. Define
$$ B(M,j) := \{f \in M:\;|f|\leq j\} \;\;\;\;\;\; |B(M,j)| := dim_{\cx}(\{f \in M:\;|f|\leq j\}),$$
which are interpreted as balls in $M$ centered at $0\in M$.
Note that the "length", $|f|$, of $f \in M$ is taken with respect to the free $R$-module $R^n$ that $M$ sits in.
The proof of Theorem 6.1 will rely on the following estimate:

\begin{lemma}
	Let $M$ be a submodule of the free $R$-module $R^n$ ($R=\cx[z_1^{\pm 1},\dotsc,z_d^{\pm 1}]$) with syzygy module
	$$ 0 \rightarrow Syz(M) \rightarrow R^{r} \rightarrow M \rightarrow 0.$$
	Then there exists $j_0$ such that for all $j > j_0$ we have the estimate
	$$ (2(j-j_0)+1)^d r - |B(Syz(M),j-j_0)| \;\;\leq\;\; |B(M,j)| \;\;\leq\;\; (2j+1)^d r - |B(Syz(M),j)|.$$
\end{lemma}	

\begin{proof}[Proof of Lemma 6.2]
	Fix a generating set $\{f^{(1)},\dots,f^{(r)}\}$ of $M$. Let $j_0 := max_{1 \leq i \leq r}|f^{(i)}|$.
	For all $j > j_0$, there are exactly $(2(j-j_0)+1)^d$ monomials $z^k$ such that $|k|\leq j-j_0$.
	For each such $k$ and for every $1\leq i \leq r$, we have $z^k f^{(i)} \in B(M,j)$.
	
	For the lower estimate, we define:
	$$U_j := Span_{\cx} \{z^k f^{(i)} | 1 \leq i \leq r, |k| \leq j-j_0\}.$$
	If $\{z^k f^{(i)} | 1 \leq i \leq r, |k| \leq j-j_0\}$ is linearly independent over $\cx$, then certainly $r (2(j-j_0)+1)^d \leq dim_{\cx} U_j \leq |B(M,j)|$.
	However, this is not true in general.
	Instead, we have relations of the form $\sum_{\alpha} c_{\alpha} z^{k_\alpha} f^{(i_{\alpha})} = 0$ where $|k_\alpha|\leq j-j_0$, and the $c_{\alpha}$ are scalars in $\cx$.
	By definition, these relations are in $1$-$1$ correspondence with the syzygies $(h_1, \dots, h_r) \in Syz(M)$ where $|h_i|\leq j-j_0$ for all $i$.
	This correspondence is easily seen to be a linear map and hence we get a Short Exact Sequence of $\cx$-vector spaces:
	$$ 0 \rightarrow B(Syz(M),j-j_0) \rightarrow B(R^{r},j-j_0) \rightarrow U_j \rightarrow 0 \;\;\; \implies $$
	$$ (2(j-j_0)+1)^d r \;-\; |B(Syz(M),j-j_0)| \;=\; dim_{\cx} B(R^{r},j-j_0) \;-\; dim_{\cx}B(Syz(M),j-j_0) $$
	$$\;=\; dim_{\cx} U_j \;\leq\; |B(M,j)|,$$
	where the second equality follows from the rank-nullity theorem.
	
	For the upper estimate, suppose that $f \in B(M,j)$.
	Since $M = \langle f^{(1)},\dots,f^{(r)} \rangle$, there exist $h_1, \dots, h_r$ s.t. $f = h_1 f^{(1)} + \dotsc + h_r f^{(r)}$.
	But $f$ consists of entries with monomials $z^k$ s.t. $|k| \leq j$ hence we may remove any monomials $z^k$ from $h_1, \dots, h_r$ with $|k| > j$ and we still get $f = h_1 f^{(1)} + \dots + h_r f^{(r)}$.
	This shows that 
	$$B(M,j) \subset W_j := Span_{\cx} \{z^k f^{(i)} | 1 \leq i \leq r, |k| \leq j\}.$$
	In the exact same manner as with the first inequality, we get a Short Exact Sequence of $\cx$-vector spaces:
	$$ 0 \rightarrow B(Syz(M),j) \rightarrow B(R^{r},j) \rightarrow W_j \rightarrow 0.$$
	Therefore, the second inequality follows by taking dimensions:
	$$ |B(M,j)| \;\leq\; dim_{\cx} U_j \;=\; dim_{\cx} B(R^{r},j) - dim_{\cx}B(Syz(M),j) \;=\; (2j+1)^d r - |B(Syz(M),j)|,$$
	and the proof of Lemma 6.2 is complete.
\end{proof}

\begin{proof}[Proof of Theorem 6.1]
	Using the lemma for each $Syz^i(K)$ for all $0\leq i \leq d$, we obtain a $j_0$ value for each. Choose $j_0$ to be the maximum out of all these values and take the following $d j_0$-thick Følner sequence:
	$$ F_j  := \mathcal{F}_j \cdot W,\;\;\; where\;\; \mathcal{F}_j:=\{k \in \z^d : |k| \leq j \}.$$
	By our choice of $F_j$ we have for all $j$:
	$$ |F_j| = |W| (2j -1)^d \;,\;\;\; \widehat{I_{F_j}l^2(V)} = B(R^{|W|},j) \;,\;\;\; dim_{\cx}\{f \in D_{\lambda}(\Gamma) : supp(f) \subset F_j\} = |B(K,j)|.$$
	By Lemma 5.1, 
	$$\nu(\{\lambda\}) = lim_{j \to \infty} \dfrac{dim_{\cx}\{f \in D_{\lambda}(\Gamma) : supp(f) \subset F_j\}}{|F_j|}
	= lim_{j \to \infty} \dfrac{|B(K,j)|}{|W| (2j + 1)^d}.$$
	By Lemma 6.2:
	$$ (2(j-j_0)+1)^d r_0 - |B(Syz(K),j-j_0)| \;\;\leq\;\; |B(K,j)| \;\;\leq\;\; (2j+1)^d r_0 - |B(Syz(K),j)|.$$
	Dividing by $(2j+1)^d$ and letting $j \to \infty$ we get
	$$lim_{j \to \infty} \dfrac{|B(K,j)|}{(2j+1)^d} = r_0 - lim_{j \to \infty} \dfrac{|B(Syz(K),j)|}{(2j+1)^d}.$$
	By induction and since $Syz^{d+1}(K) = 0$,
	$$ lim_{j \to \infty} \dfrac{|B(K,j)|}{(2j+1)^d} = r_0 - r_1 + r_2 - \dotsc + (-1)^d r_d.$$
	and the theorem follows.
\end{proof}		

\textbf{Now let $R=\mathbb{C}G$ where $G$ is a finitely generated group of subexponential growth.}
A group has subexponential growth whenever the growth function $\gamma_S(n)=\{g \in G: |g|_S=n\}$ with respect to any (equivalently all) finite generating sets $S$ of $G$ is a sequence of subexponential growth ($|g|_S$ is the distance of $g$ from $1$ along the Cayley graph $\Gamma(G,S)$).
Groups of subexponenitial growth are always amenable \cite{GGTLoeh}.
We finish this section with a generalization of Theorem 6.1.
The author believes that further study is needed.
A key obstacle is the lack of a Hilbert Syzygy theorem to group algebras of non-abelian groups.
Nonetheless, whenever the module of finite support eigenfunctions admits a finite resolution by finitely generated free $\mathbb{C}G$-modules, we obtain the same conclusion as in Theorem 6.1.

\begin{theorem}
	Suppose that $\Gamma$ is a $G$-periodic graph where $G$ is a finitely generated group of subexponential growth and $\lambda$ is an eigenvalue of $\Delta$ on $\Gamma$ (or any periodic difference operator $D$ of finite order).
	Let $K$ be the $\mathbb{C}G$-module of finite support eigenfunctions of $\lambda$.\\
	\textbf{If} $K$ admits a finite resolution by finitely generated free $R$-modules ($R=\mathbb{C}G$)
	$$0 \rightarrow R^{r_d} \rightarrow R^{r_{d-1}} \;\dotsc\; \rightarrow R^{r_1} \rightarrow R^{r_0} \rightarrow K \rightarrow 0.$$		
	\textbf{Then} the following formula about the density of $\{\lambda\}$ holds:
	$$ \nu(\{\lambda\}) = \dfrac{1}{|\Gamma / G|} \sum_{k=0}^{d} (-1)^k r_k.$$
\end{theorem}

\begin{proof}
	Note first that the hypothesis of Lemma 5.1 is satisfied since $r_0 < \infty$.
	Since $G$ has subexponential growth, fixing any generating set $S$ of $G$, the balls of radius $n$, $B(n)=\{g: |g|_S\leq n\}$ have a subsequence $F_j:=B(n_j)$ which is a $k$-thick Folner sequence for all $k \in \n$.
	This is a standard argument for groups of subexponential growth and can be found, for instance, in \cite{GGTLoeh}.
	
	For each $g\in G$, let $|g|:=min\{j:g\in B(n_j)\}$.
	For each $f_1 = \sum_g c_g g \in \mathbb{C}G$, let $|f_1|:=max\{|g|:c_g\neq 0\}$.
	Next, for each $f=(f_1,\dotsc,f_r)$ let $|f|:=max\{|f_1|,\dotsc,|f_r|\}$.
	For each $M$ is a submodule of $R^n$ let
	$$ B(M,j) := \{f \in M:\;|f|\leq j\}, \;\;\;\;\;\; |B(M,j)| := dim_{\cx}(\{f \in M:\;|f|\leq j\}),$$
	and finally, $\gamma(n)=|B(n)|$ is simply the growth function with respect to $S$.
	Similar to Lemma 6.2, one can show that 
	then there exists $j_0$ such that for all $j > j_0$ we have the estimate
	$$ \gamma(n_{j}-n_{j_0}) r - |B(Syz(M),j-j_0)| \;\;\leq\;\; |B(M,j)| \;\;\leq\;\; \gamma(n_{j}) r - |B(Syz(M),j)|,$$
	where $Syz(M)$ is the syzygy module with respect to a finite generating set of $M$ of size $r<\infty$.
	Finally, using the same telescoping argument as before along with the fact that for arbitrarily large $k \in \n$:
	$$ lim_{j \to \infty} \dfrac{\gamma(n_{j}-k)}{\gamma(n_{j})} = lim_{j \to \infty} \dfrac{\gamma(n_{j})}{\gamma(n_{j}+k)} = 1 - lim_{j \to \infty} \dfrac{|\partial_{k} F_j|}{|\partial_{k} F_j \cup F_j|} = 1-0=1,$$
	the theorem follows.
\end{proof}

\section{An example with nontrivial syzygy}

In this section, we provide an example of a $\z^3$-periodic graph $\Gamma$ such that $\lambda=-2$ is an eigenvalue of the adjacency operator $A$ on $\Gamma$, and the $\cx[x^{\pm 1}, y^{\pm 1}, z^{\pm 1}]$-module $K$ of finite support eigenfunctions is not free.
We find a generating set for $K$ and its first syzygy module using standard techniques from commutative algebra (the details of the computation are shown in Appendix B).
We then apply Theorem 1.1 to compute $\nu(\{-2\})$.
This example indicates that our analysis in Section 6 and Theorem 1.1 has mathematical content (if $K$ was always free, then these results would be pointless).

A simple example of a periodic graph with non-empty pure-point spectrum is the Kagome lattice.
It was studied in \cite{KagomeAndFiniteSupport} by Lenz, Peyerimhoff, Post and Veseli\'{c}, where they found a spanning set for all finite support eigenfunctions of the only eigenvalue of the Laplacian $\Delta$.
In particular, from their work it follows that the module of finite support eigenfunctions is generated by a single element (hence it is free).

Before we describe our example, let us first describe its natural analogue in two dimensions (to aid with visualizing the three dimensional example).
Define the sets $W:=\{(1,0),(0,1),(-1,0),(0,-1)\} \subset \real^2$
and $V := W + 4\z^2$, where $'+'$ denotes the usual sumset operation.
The edges of $\Gamma=(V,E)$ are defined by
$$ u \sim v \iff ||u-v||\leq 2,$$
where $u,v \in V$ and $||\cdot||$ is the standard Euclidean norm.
The group action $\z^2 \curvearrowright V$ is given by
$$(a,b)\cdot (m,n) := (m+4a, n+4b),$$
for all $(a,b)\in\z^2$ and $(m,n)\in V$.
It is easy to see that $\Gamma$ is a $\z^2$-periodic graph with fundamental domain $W$, and that all the vertices in $W$ are connected with each other.
An illustration of $\Gamma$ is shown in Figure 1.
Note that $e_1$ and $e_2$ denote the horizontal and vertical translations of $\z^2 \curvearrowright V$.

\begin{figure}[h]
	\begin{center}
		\includegraphics[scale=0.15]{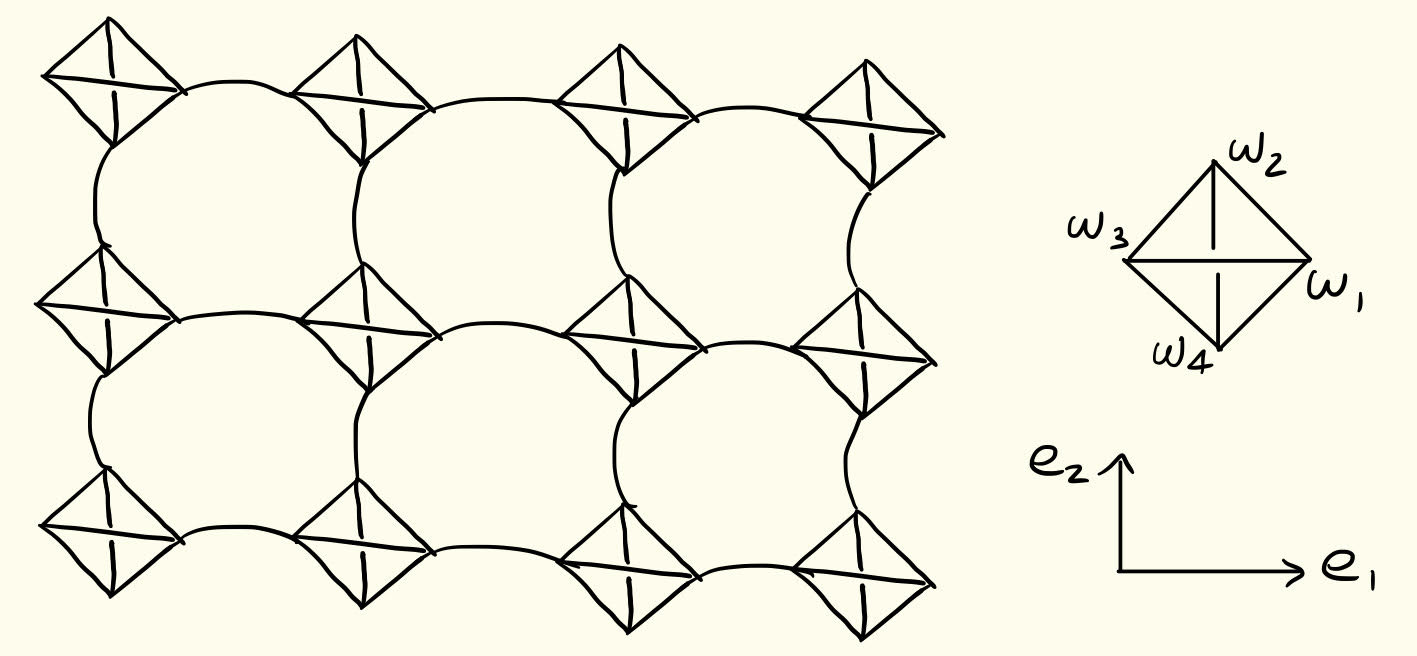}
		\caption{The two dimensional example}
	\end{center}
\end{figure}	

Consider the adjacency operator $A$ on $l^2(V)$ defined in Section 2.
If we order the vertices in $W$ as shown in Figure 1, the Floquet-Bloch transform of $A$ is given by the following $4\times 4$ matrix with entries in $\cx[x^{\pm 1}, y^{\pm 1}]$:
$$ \widehat{A} = 	
\begin{pmatrix}
	0			& 1			& 1+x 	& 1	\\
	1			& 0 		& 1 	& 1+y	\\
	1+x^{-1}	& 1			& 0 	& 1		\\
	1			& 1+y^{-1}	& 1 	& 0		\\	
\end{pmatrix},$$
where the $i^{th}$ row corresponds to all the neighbors of $w_i\in W$.
Figure 2 shows an eigenfunction of finite support of the eigenvalue $\lambda = -2$. The filled red circles are the vertices where the function takes value $+1$, the hollow red circles for $-1$, and the function is zero on the rest of the vertices. Its Floquet-Bloch transform is
$$\begin{pmatrix}
	-x+xy\\
	y-xy\\
	1-y\\
	-1+x
\end{pmatrix}.$$
By performing an analysis similar to the one we will carry for the three dimensional example, one can show that this eigenfunction generates the $\cx[x^{\pm 1}, y^{\pm 1}]$-module of finite support eigenfunctions of $\lambda = -2$.

\begin{figure}[h]
	\begin{center}
		\includegraphics[scale=0.18]{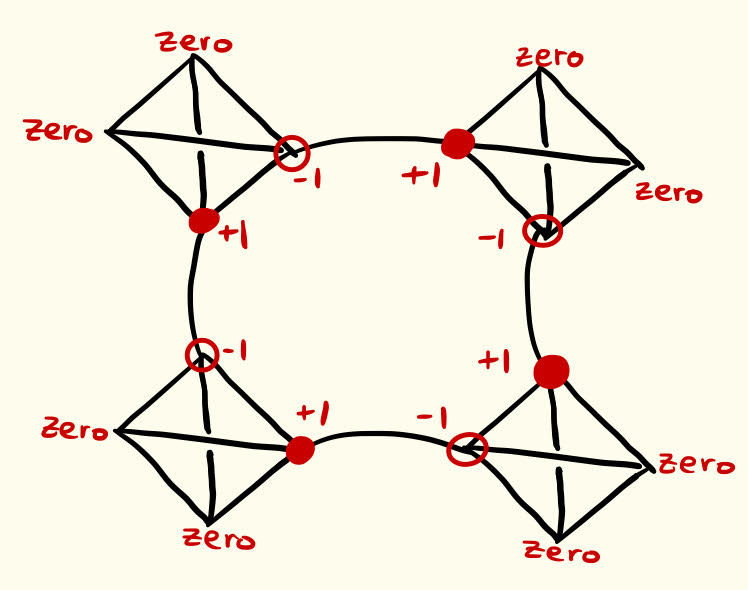}
		\caption{The generator of the module of finite support eigenfunctions}
	\end{center}
\end{figure}

We now turn to our example of interest.
Define the sets 
$$W:=\{(1,0,0),(0,1,0),(0,0,1),(-1,0,0),(0,-1,0),(0,0,-1)\},\;\;V:=W+4\z^3,$$
the edges $E$ of $\Gamma = (V,E)$:
$ u \sim v \iff ||u-v||\leq 2$ for all $u,v,\in V,$
and the group action $\z^3 \curvearrowright V$ by
$$(a,b,c) \cdot (m,n,p) := (m+4a,n+4b,p+4c),$$
where $(a,b,c)\in \z^3$ and $(m,n,p) \in V$.
Then $\Gamma$ is a $\z^3$-periodic graph with fundamental domain $W$, and all the vertices in $W$ are connected to each other.
Figure 3 attempts to illustrate $\Gamma$.
To the left, we have $8$ translated copies of the fundamental domain $W$ and the vibrating lines are the edges in $\Gamma$ (in addition to the edges that connect vertices in the same domain, which are not drawn here).
To the right, we have the induced subgraph of $\Gamma$ on $ [0,4]^3 \cap V$ (the vertices are in red and the edges in blue).

\begin{figure}[h]
	\begin{center}
		\includegraphics[scale=0.15]{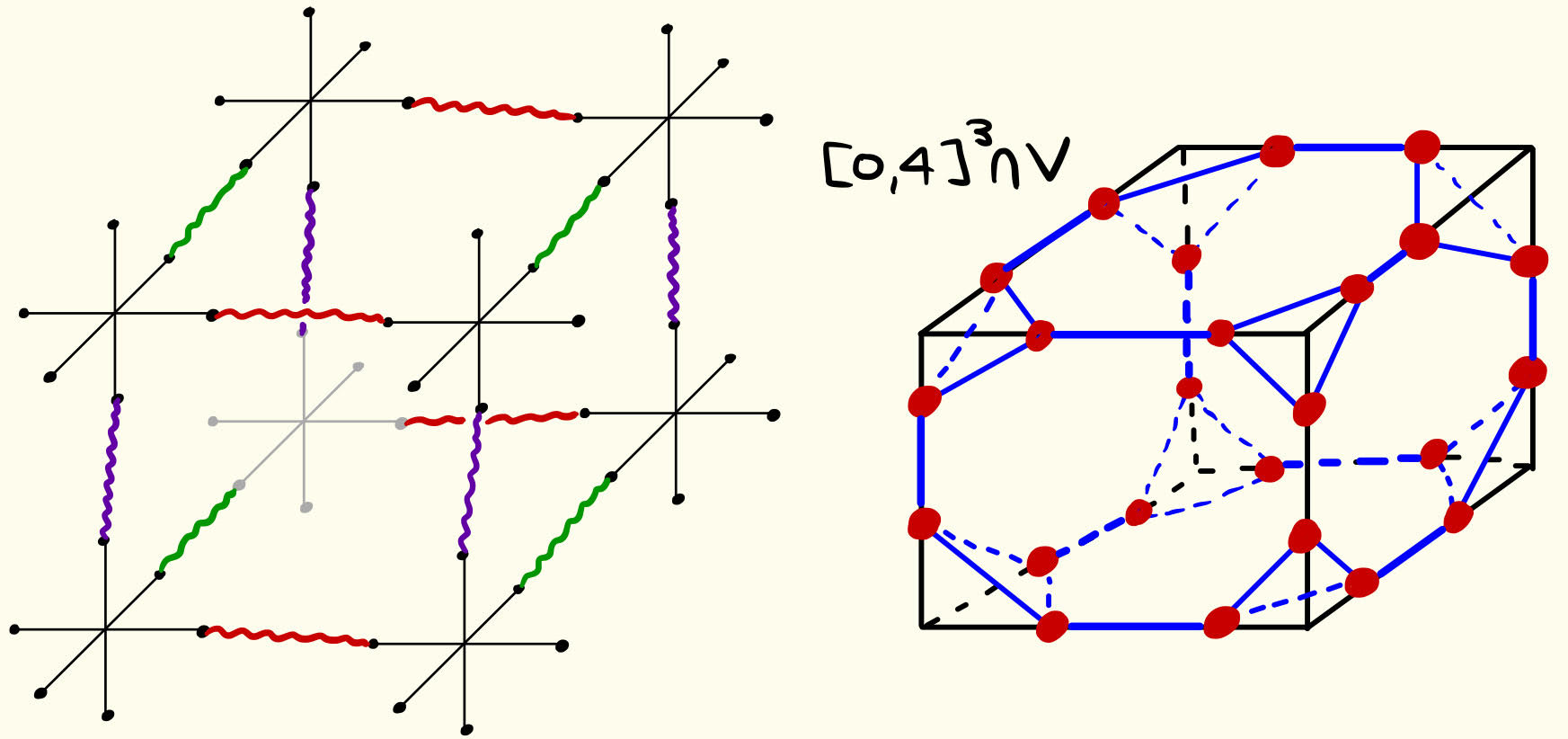}
		\caption{A $\z^3$-periodic graph with nontrivial syzygies}
	\end{center}
\end{figure}

Next, Figure 4 shows the $6$ vertices in $W$ and labels them $\{w_1,w_2,w_3,w_4,w_5,w_6\}$, so that we can write down a matrix representation for $\widehat{A}$.
Also shown in Figure 4 are the translations $e_1, e_2, e_3$ which generate the group action, and become multiplication by the monomials $x^{-1}$, $y^{-1}$, and $z^{-1}$ under the Floquet-Bloch transform.
The matrix representation of $\widehat{A}$ is given by:
$$\widehat{A}=\begin{pmatrix}
	0		& 1			& 1			& 1+x	& 1		& 1		\\
	1		& 0			& 1			& 1		& 1+y	& 1		\\
	1+x^{-1}& 1			& 0			& 1		& 1		& 1+z	\\
	1		& 1+y^{-1}	& 1			& 0		& 1		& 1		\\
	1		& 1			& 1			& 1		& 0		& 1		\\
	1		& 1			& 1+z^{-1}	& 1		& 1		& 0		\\
\end{pmatrix}.$$

\begin{figure}[h]
	\begin{center}
		\includegraphics[scale=0.14]{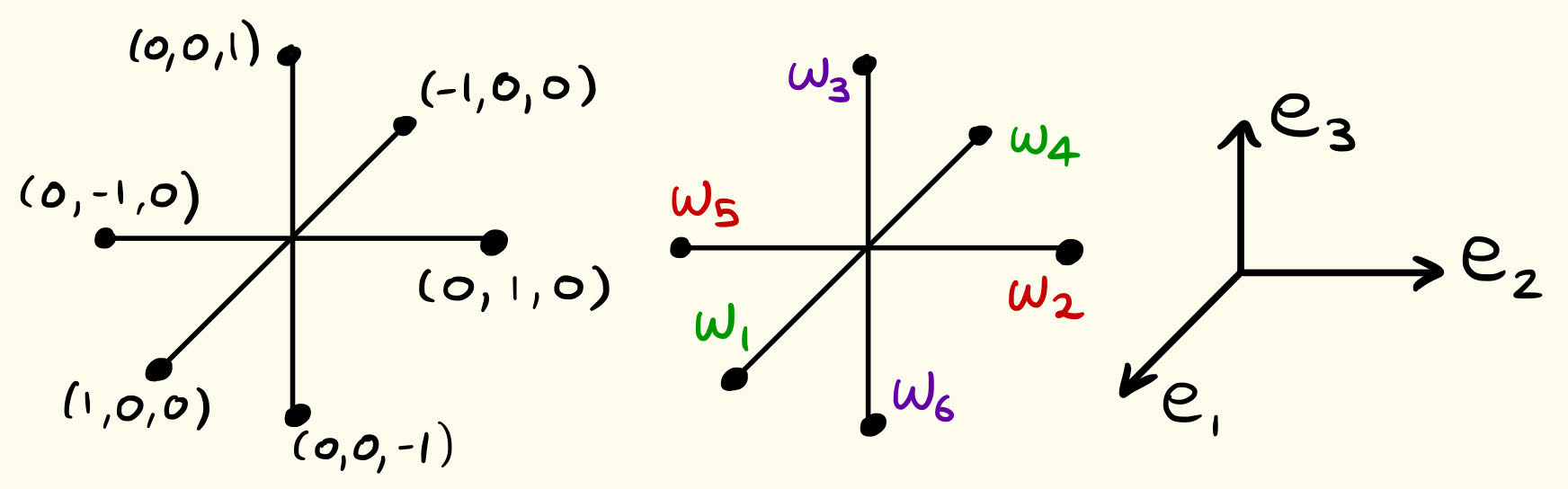}
		\caption{Labeling the fundamental domain and the group action}
	\end{center}
\end{figure}

Notice that $\lambda = -2$ is an eigenvalue of $A$, since $f^{(1)},f^{(2)},$ and $f^{(3)}$ shown in Figure 5 are finite support eigenfunctions (once again filled read means a value of $+1$ and hollow red means a value of $-1$).

\begin{figure}[h]
	\begin{center}
		\includegraphics[scale=0.17]{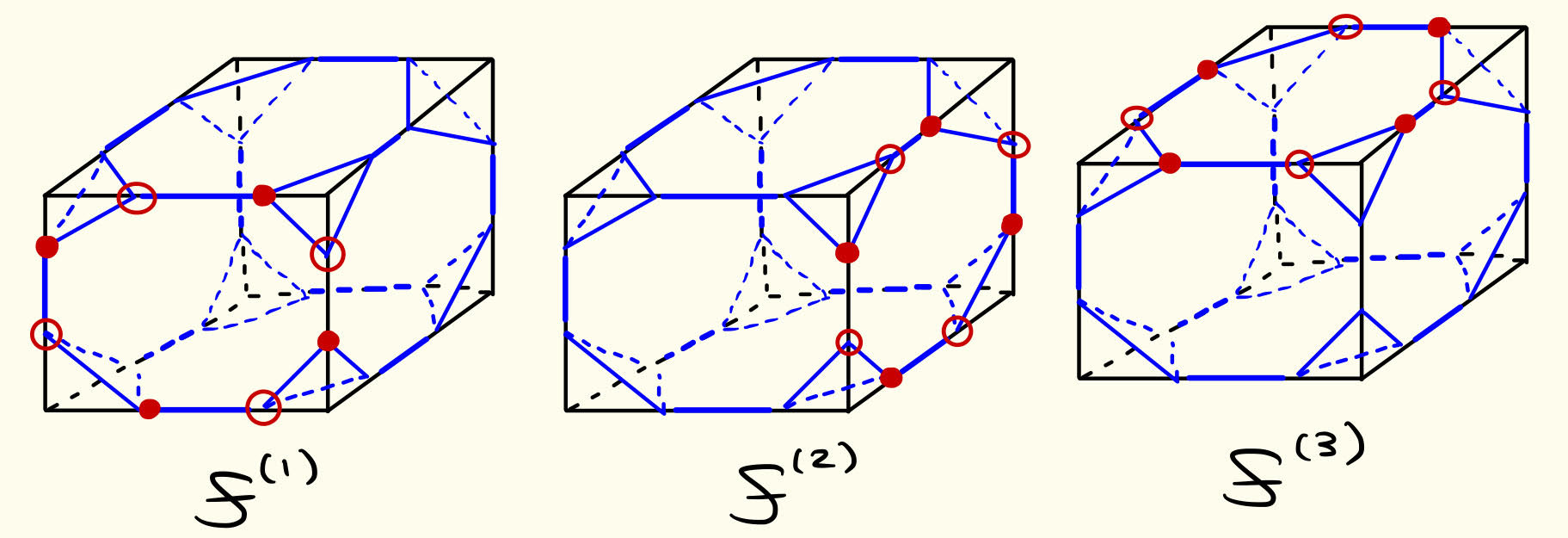}
		\caption{Generators of the module of finite support eigenfunctions}
	\end{center}
\end{figure}

Now, translate $f^{(1)},f^{(2)},$ and $f^{(3)}$ by $-e_1, -e_2$ and $-e_3$ respectively and multiply each by $-1$.
We get the eigenfunctions shown in Figure 6.
The sum of all six eigenfunctions sums up to zero.
After taking the Floquet-Bloch transform, this sum becomes the following syzygy relation:
$$(1-x)\widehat{f}^{(1)} + (1-y)\widehat{f}^{(2)} + (1-z)\widehat{f}^{(3)} = 0.$$

\begin{figure}[h]
	\begin{center}
		\includegraphics[scale=0.17]{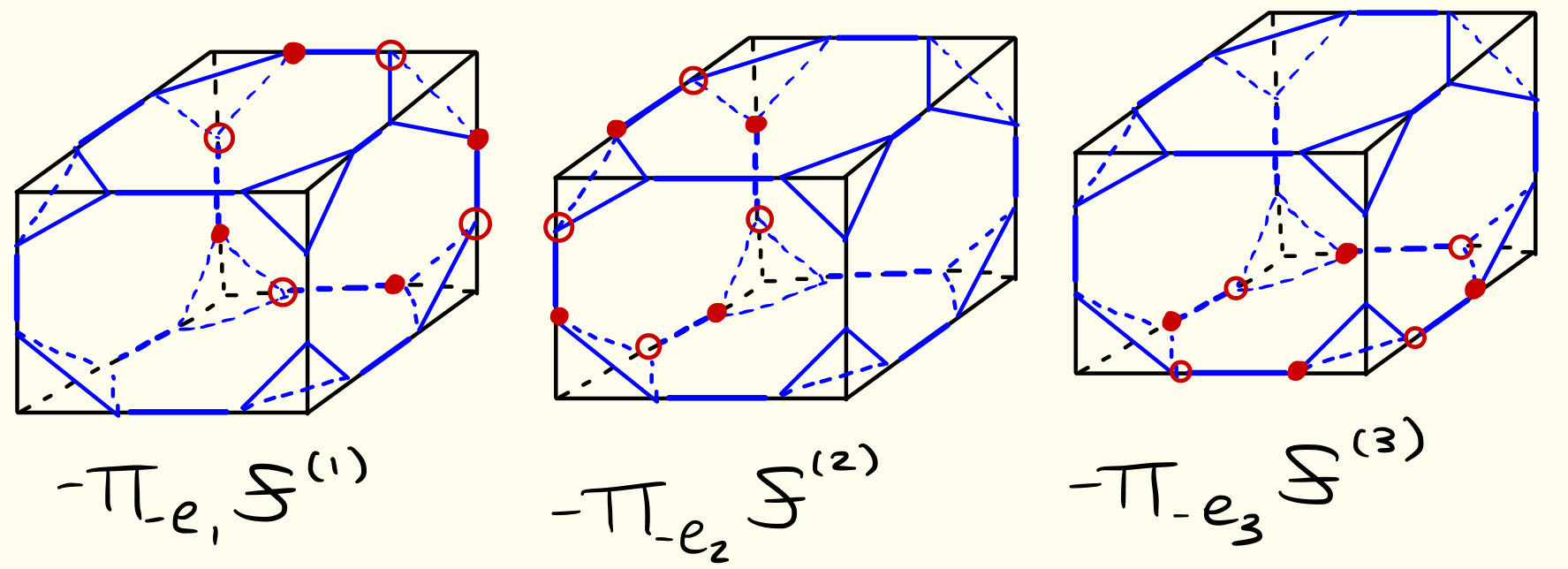}
		\caption{Translated copies of $f^{(1)},f^{(2)},$ and $f^{(3)}$}
	\end{center}
\end{figure}

\begin{proposition}
    The module of finite support eigenfunctions $K$ of $\lambda = -2$ is generated by $f^{(1)},f^{(2)},$ and $f^{(3)}$, and its first syzygy module is generated by $(1-x, 1-y, 1-z)$.
	Equivalently, we have a free resolution of K:
	$$0 \rightarrow \cx[x^{\pm1},y^{\pm1},z^{\pm1}] \rightarrow \bigoplus_{i=1}^3 \cx[x^{\pm1},y^{\pm1},z^{\pm1}] \rightarrow K \rightarrow 0,$$
	where the first map is $h \mapsto (h(1-x), h(1-y), h(1-z))$
	and the second map is $(h_1, h_2, h_3) \mapsto h_1 f^{(1)} + h_2 f^{(2)} + h_3 f^{(3)}$.
	In particular, by Theorem 1.1, $\nu(\{-2\}) = (1/6)(3-1) = 1/3$.
\end{proposition}

The proof of Proposition 7.1 is a long computation using standard commutative algebra techniques, and is included in Appendix B.
Observe that we have $3$ eigenfunctions up to translations and linear combinations and $6$ elements in the fundamental domain, so one could naively guess that $\nu(\{-2\})$ is $3/6=1/2$, which is a wrong guess.
Why does this intuition fail?
The functions $f^{(1)},f^{(2)},$ and $f^{(3)}$ are linearly independent, but the set of all translations of $f^{(1)},f^{(2)},$ and $f^{(3)}$ is linearly dependent.
How many dependence relations are there?
They are precisely given via the syzygy module.
We finish this section with the following conjecture generalizing the two preceding examples:

\textbf{Conjecture:} 
	Consider the standard basis $(\delta_{1i})_i, ... (\delta_{di})_i$ of $\z^d$ and the sets 
	$$W:=\{(\delta_{1i})_i, ... (\delta_{di})_i, (-\delta_{1i})_i, ... (-\delta_{di})_i\}, \;\;and \;\;V:=W+4\z^d.$$
	We define a graph $\Gamma$ by connecting $(u,v)\in V\times V$ whenever $||u-v||\leq 2$, and the group action $\z^d \curvearrowright V$ via:
	$$(a_1,\dots a_d)\cdot (v_1,\dots v_d) := (v_1+4 a_1, \dots v_d+4a_d),$$
	where $(a_1,\dots a_d)\in\z^d$ and $(v_1,\dots v_d)\in V$.
	Then $\lambda=-2$ is an eigenvalue of the adjacency operator $A$ on $\Gamma$ and has density $(d-1)/(2d)$.
	Moreover, the module of finite support eigenfunctions $K$ admits a free resolution of the form:
	$$0 \rightarrow R^{d\choose d} \;\dots\; \rightarrow R^{d\choose4} \rightarrow R^{d\choose 3} \rightarrow R^{d\choose2} \rightarrow K \rightarrow 0,$$
	where $R = \cx[z_1^{\pm 1},\dots,z_d^{\pm 1}]$, and the kernel of each map is not free (with the exception of the map $R^{d\choose d} \rightarrow R^{d\choose d-1}$).

\section*{Appendix A: Proofs of Lemma 2.4, Theorem 2.6, Theorem 2.7, and Lemma 5.1}

\begin{proof}[Proof of Lemma 2.4 (Thick Følner sequences)]
	The case when $r=1$ and $\Gamma = \Gamma(G,S)$ follows by the definition of amenability.
	Next, lets look at the case where $r>1$ and $\Gamma = \Gamma(G,S)$ .
	Consider the generating set $S':=S \cup S^2 \cup \dotsc S^r$
	where $S^r := \{s_1^{\epsilon_1}\dotsc s_r^{\epsilon_r} : s_1, \dotsc s_r \in S,\;\epsilon_1, \dotsc \epsilon_r \in \{+1,-1\}\}$.
	Since $G$ is amenable, there exists a 1-thick Følner sequence $\{F_j\}$.
	By construction, the $1$-thick boundary of $F_j$ with respect to $\Gamma(G,S^r)$
	is precisely the $r$-thick boundary of $F_j$ with respect to $\Gamma(G,S)$
	therefore $\{F_j\}$ is our desired Følner sequence with $l=r$.
	
	Next, we have the general case. 
	Fix a fundamental domain $W$ and a generating set $S$ of $G$.
	Each element $w \in W$ is connected to finitely many vertices $v \in V =  \sqcup_{g\in G} g \cdot W$.
	So for each such $v$, there exists a unique $g_v \in G, w_v \in W$ such that $v = g_v \cdot w_v$.
	Consider $|g_v|$, which is the distance of $g$ from $1$ in the Cayley graph $\Gamma(G,S)$.
	Out of all $v \sim w$ and all $w \in W$ pick the largest value of $|g_v|$ and call it $t$.
	Inductively, it follows that if $u \in g_u \cdot W, \;\; v \in g_v \cdot W$ and $d(u,v)\leq r$ then $|g_u g^{-1}_v| \leq r t$.
	It follows that for each $\mathcal{F} \subset G$, $\partial_r (\mathcal{F}\cdot W) \subset (\partial_l \mathcal{F})\cdot W$.
	By the previous case, we may choose an $l := r t$ thick Følner sequence of $\Gamma(G,S)$, $\{\mathcal{F}_j\}$
	and define $F_j := \mathcal{F}_j \cdot W$.
	We have:
	$$ \dfrac{|\partial_{r} F_j|}{|F_j|} \leq \dfrac{|(\partial_{r t}\mathcal{F}_j) \cdot W|}{|\mathcal{F}_j \cdot W|} =  \dfrac{|\partial_{r t}\mathcal{F}_j| |W|}{|\mathcal{F}_j| |W|} = \dfrac{|\partial_{r t}\mathcal{F}_j| }{|\mathcal{F}_j| }\to 0,$$
	hence $\{F_j\}_j$ is a standard $r$-thick Følner sequence of $\Gamma$.
\end{proof}

\begin{proof}[Proof of Theorem 2.6 (Strong Localization of Eingenfunctions, Kuchment-Veseli\'{c})]
	We outline the proof of Higuchi and Nomura from \cite{Higuchi:Nomura:Paper}.
	Since $\lambda$ is an eigenvalue, there exists $f \in l^2(V)\setminus \{0\}$ such that $\Delta f = \lambda f$.
	Without loss of generality $||f|| = 1$.
	Fix a fundamental domain $W$.
	Since $f \neq 0$ and translations of eigenfunctions are still eigenfunctions, without loss of generality $f(w)\neq 0$ for some $w\in W$.
	Picking an orthonormal basis $\{\phi_i\}_i$ of $l^2(V)$ such that $\phi_1 = f$ we have:
	$$ \nu(\{\lambda\}) = \dfrac{1}{|W|} tr(E(\{\lambda\})I_W) \geq \dfrac{\langle I_W E(\{\lambda\})f,f\rangle}{|W|}  = \dfrac{\langle I_W f,f \rangle}{|W|} = \dfrac{\sum_{w \in W} |f(w)|^2}{|W|} > 0.$$
	Next, since the action of $G$ commutes with $\Delta$, for all $g \in G$
	$$ |W| \nu(\{\lambda\}) = tr(E(\{\lambda\})I_W) = tr(E(\{\lambda\})\pi_{g} \pi_{g^{-1}}  I_W) = tr(E(\{\lambda\})I_{g \cdot W}).$$
	And hence for all finite subsets $\mathcal{F}$ of $G$, $\nu(\{\lambda\}) = \dfrac{1}{|\mathcal{F} \cdot W|} tr(E(B)I_{\mathcal{F}\cdot W})$.
	
	Now, pick a standard $l$-thick Følner sequence $\{F_j\}_j$, where $F_j = \mathcal{F}_j \cdot W$ for all $j$ and $l\in \n$ is the corresponding constant from Lemma 2.4 with $r=2$.
	Also, pick an orthonormal basis $\{\phi_1,\dotsc \phi_{m_j}\}$ of $I_{(\mathcal{F}_j\cup \partial_l \mathcal{F}_j)\cdot W}E(\{\lambda\})l^2(V)$
	and extend it to an orthonormal basis $\{\phi_j\}_{j=1}^{\infty}$ of $l^2(V)$.
	We have:
	$$ tr(I_{(\mathcal{F}_j\cup \partial_l \mathcal{F}_j)\cdot W}E(\{\lambda\}) = \sum_{i=1}^{m_j}\langle I_{(\mathcal{F}_j\cup \partial_l \mathcal{F}_j)\cdot W}E(\{\lambda\})\phi_i,\phi_i\rangle + \sum_{i=m_j+1}^{\infty}\langle I_{(\mathcal{F}_j\cup \partial_l \mathcal{F}_j)\cdot W}E(\{\lambda\})\phi_i,\phi_i\rangle$$
	$$ = \sum_{i=1}^{m_j}\langle I_{(\mathcal{F}_j\cup \partial_l \mathcal{F}_j)\cdot W}\phi_i,\phi_i\rangle \;\leq m_j.$$
	Using the above estimate, we claim that there exists $j$ such that $|\partial_l \mathcal{F}_j\cdot W| < m_j$.
	If not, that $|\partial_l \mathcal{F}_j\cdot W| \geq m_j$ for all $j$ and hence
	$$ 0 < \nu(\{\lambda\}) = \dfrac{1}{|(\mathcal{F}_j\cup \partial_l \mathcal{F}_j)\cdot W|} tr(I_{(\mathcal{F}_j\cup \partial_l \mathcal{F}_j)\cdot W}E(\{\lambda\}) \leq \dfrac{m_j}{|(\mathcal{F}_j\cup \partial_l \mathcal{F}_j)\cdot W|} \leq \dfrac{|\partial_l \mathcal{F}_j\cdot W|}{|F_j|} \to 0,$$
	a contradiction. 
	Picking $j$ such that $|\partial_l \mathcal{F}_j\cdot W| < m_j$,
	it follows that $\{I_{\partial_l \mathcal{F}_j\cdot W}\phi_1,\dotsc I_{\partial_l \mathcal{F}_j\cdot W}\phi_{m_j}\}$
	is a linearly dependent set so we can find $a_1, \dotsc a_{m_j}$ not all zero such that
	$$h:=\sum_{i=1}^{m_j} a_i \phi_i \equiv 0 \;\; on \; \partial_l \mathcal{F}_j\cdot W.$$		
	By Lemma 2.4, we chose $l$ so that $\partial_2 F_j \subset \partial_l \mathcal{F}_j\cdot W$. It follows that $I_{F_j} h$ is an eigenfunction with finite support inside $F_j$ which is nonzero due to the independence of the $\phi_i$ on $F_j$.
\end{proof}

\begin{proof}[Proof of Theorem 2.7 (Finite Support Approximation of Eigenfunctions, Kuchment-Veseli\'{c})]
	Let $\mathcal{M}$ be the closure of $D_{\lambda}(\Gamma)$ in $l^2(V)$ and $E_{\lambda}$ be the eigenspace of $\lambda$.
	Suppose the contrary, i.e. $\mathcal{M} \neq E_{\lambda}$.
	Then the orthogonal complement of $\mathcal{M}$ (with respect to the subspace $E_\lambda$) is non-trivial, $\mathcal{N} := \mathcal{M}^\perp \cap E_{\lambda} \neq \{0\}$.
	Note that since for all $g \in G, \;\; \pi_g D(\Gamma) = D(\Gamma)$,
	$$ f \in \mathcal{N} \iff \;for\;all\; h \in D_{\lambda}(\Gamma),\; \langle f,h \rangle =0 \iff $$
	$$ \;for\;all\; h \in D_{\lambda}(\Gamma),\; \langle \pi_g f,\pi_g  h\rangle =0 \iff \;for\;all\; h \in D_{\lambda}(\Gamma),\; \langle \pi_g f,h \rangle=0.$$
	Hence $\mathcal{N}$ is also invariant under translations.
	
	Next, consider the orthogonal projection onto $\mathcal{N}$, $P_{\mathcal{N}}$ and fix a fundamental domain $W$.
	Define the "density" of $\mathcal{N}$ as 
	$$ \nu:= \dfrac{1}{|W|} tr(I_W P_{\mathcal{N}})$$
	where the trace is taken with respect to $l^2(V)$.
	Pick some $f\in\mathcal{N}$ with $||f||=1$ and translated so that $f(w) \neq 0$ for some $w \in W$.
	Then by extending $\{f\}$ to an orthonormal basis, we conclude that $\nu > 0$.
	Since $\mathcal{N}$ is invariant under translations, it also follows that $\nu = \dfrac{1}{|\mathcal{F} \cdot W|} tr(I_{\mathcal{F}\cdot W} P_{\mathcal{N}})$ for every finite subset $\mathcal{F}$ of $G$.
	
	Now, pick a standard $l$-thick Følner sequence $\{F_j\}_j$, where $F_j = \mathcal{F}_j \cdot W$ for all $j$ and $l\in \n$ is the corresponding constant from Lemma 2.4 with $r=2$.
	For all $j$, pick an orthonormal basis $\{\psi_1,\dotsc \psi_{n_j}\}$ of $I_{(\mathcal{F}_j\cup \partial_l \mathcal{F}_j)\cdot W}P_{\mathcal{N}} l^2(V)$,
	and extend it to an orthonormal basis $\{\psi_j\}_{j=1}^{\infty}$ of $l^2(V)$.
	We conclude that $tr(I_{(\mathcal{F}_j\cup \partial_l \mathcal{F}_j)\cdot W} P_{\mathcal{N}}) \leq n_j$.
	
	Similar to before, if for all $j$, $|\partial_l \mathcal{F}_j\cdot W| \geq n_j$, then $0 < \nu < \dfrac{|\partial_l \mathcal{F}_j\cdot W|}{|F_j|} \to 0$
	This cannot happen, so there is some $j$ such that $|\partial_l \mathcal{F}_j\cdot W| \leq n_j$.
	We can then find $b_1, \dotsc b_j$ not all zero such that 
	$$ h:=\sum_{i=1}^{n_j} b_i \psi_i \equiv 0 \;\; on \; \partial_l \mathcal{F}_j\cdot W.$$
	By Lemma 2.4, we chose $l$ so that $\partial_2 F_j \subset \partial_l \mathcal{F}_j\cdot W$. It follows that $h$ is a finite support eigenfunction, and by definition we get: $<h,f> = 0 $ for all $f \in \mathcal{N}$.
	
	Here comes the contradiction.
	For all $i=1\dotsc n_j$, $\psi_i \in I_{(\mathcal{F}_j\cup \partial_l \mathcal{F}_j)\cdot W}P_{\mathcal{N}} l^2(V) = I_{(\mathcal{F}_j\cup \partial_l \mathcal{F}_j)\cdot W} \mathcal{N}$, so we can extend $\psi_i$ to $\bar{\psi_i} \in \mathcal{N}$ such that $I_{(\mathcal{F}_j\cup \partial_l \mathcal{F}_j)\cdot W} \bar{\psi_i} = \psi_i$.
	Notice that $\sum_{i=1}^{n_j} b_j \bar{\psi_i} \in \mathcal{N}$.
	We have:
	$$0 = \;\langle h, \sum_{i=1}^{n_j} b_j \bar{\psi_i}\rangle \; = \sum_{v \in F_j} |h(v)|^2 \neq 0.$$
	This is a contradiction, therefore $\mathcal{N} = \emptyset \implies \mathcal{M} = E_{\lambda}$ and the claim follows.		
\end{proof}

\begin{proof}[Proof of Lemma 5.1 (Finite Support Approximation for the density of an eigenvalue)]
	Suppose that $\widehat{D_{\lambda}(\Gamma)}$ is generated by $\{\hat{f}^{(1)}, \dotsc, \hat{f}^{(r)}\}$.
	Consider the supports of $f^{(1)}, \dotsc, f^{(r)}$ which are all finite, 
	hence we can find a finite $K\subset G$ such that 
	$\cup_{i=1}^r supp(f^{(i)}) \subset \cup_{g \in K} g \cdot W$ and let $j_0 := 2 max_{g \in K} |g|$.
	Take a standard $j_0$-thick Følner sequence $\{F_j\}_j$.
	As we saw in the proof of Theorem 2.6, $\nu(\{\lambda\}) = \dfrac{1}{|F_j|} tr(E(\{\lambda\})I_{F_j})$.
	Define the $\cx$-vector spaces
	$$U_j := Span\{ \pi_g f^{(i)} : g \in G,\; i=1,\dotsc,r \; supp(\pi_g f^{i})\subset F_j\},$$
	$$W_j := Span\{ \pi_g f^{(i)} : g \in G,\; i=1,\dotsc,r \; supp(\pi_g f^{i})\subset F_j \cup \partial_{j_0} F_j\}.$$
	By Theorem 2.7, the (closed) subspace of $l^2(V)$ generated by $\{ \pi_g f^{(i)} : g \in G,\; i=1,\dotsc,r\}$ is precisely $E_\lambda$.
	By our choice of $j_0$, for all $\pi_g f^{(i)} \in \{\pi_g f^{(i)} : supp(\pi_g f^{i}) \not\subset F_j \cup \partial_{j_0} F_j, 1 \leq i \leq r\}$  we have that
	$supp(\pi_g f^{(i)}) \cap F_j = \emptyset$,
	hence $\langle I_{F_j} \pi_g f^{(i)}, \pi_g f^{(i)}\rangle =0$ and we conclude that $tr(E(\{\lambda\})I_{F_j}) \leq dim_{\cx} W_j$.
	On the other hand, each element $\phi$ of an orthonormal basis for $U_j$ satisfies $\langle I_{F_j} \phi, \phi \rangle = ||\phi||^2= 1$.
	We end up with the estimate:
	$$ dim_{\cx} U_j \leq tr(E(\{\lambda\})I_{F_j}) \leq dim_{\cx} W_j.$$
	Next, we work on estimating $dim_{\cx}\{f \in E_{\lambda} : supp(f) \subset F_j\}$.
	Obviously $U_j \subset \{f \in E_{\lambda} : supp(f) \subset F_j\}$.
	On the other hand, if $f \in E_{\lambda}$ with $supp(f) \subset F_j$, $f$ is the linear combination of translations of $f^{(1)}, \dots, f^{(r)}$. 
	The values of $f$ on $F_j$ depend only through terms $\pi_g f^{(i)}$ whose support is inside $F_j \cup \partial_{j_0} F_j$, by the construction of $j_0$.
	Hence there is a function $f' \in W_j$ which is equal to $f$ on $F_j$.
	We get a map from $\{f \in E_{\lambda} : supp(f) \subset F_j\}$ to $W_j$ sending $f$ to $f'$ which is obviously injective.
	We conclude that for all $j$:
	$$ dim_{\cx} U_j \leq dim_{\cx} \{f \in E_{\lambda} : supp(f) \subset F_j\} \leq dim_{\cx} W_j.$$
	Finally, we ask: what is $q_j := |\{ \pi_g f^{i} : g \in G,\; i=1,\dots,r \; supp(\pi_g f^{i})\not\subset F_j \;but\; supp(\pi_g f^{i})\subset F_j \cup \partial_{j_0} F_j\}|$? We get the obvious bound $q_j \leq |\partial_{j_0} \mathcal{F}_j | r $.
	Dividing by $|F_j|$, we take the limit as $j \to \infty$:
	$$0 \leq lim_{j \to \infty} \dfrac{dim_{\cx} W_j - dim_{\cx} U_j}{|F_j|} \leq lim_{j \to \infty} \dfrac{q_j}{|F_j|} \leq lim_{j \to \infty} \dfrac{|\partial_{j_0} \mathcal{F}_j| r}{|F_j|} = 0.$$
	Therefore, by squeezing between $U_j$ and $W_j$: $$\nu(\{\lambda\}) = \dfrac{1}{|F_j|} tr(E(\{\lambda\})I_{F_j}) = lim_{j \to \infty} \dfrac{dim_{\cx}\{f \in D_{\lambda}(\Gamma) : supp(f) \subset F_j\}}{|F_j|},$$
	and the proof is complete.
\end{proof}

\section*{Appendix B: Computation for Proposition 7.1}

To find all eigenfunctions of finite support is equivalent to solving the following system of linear equations in $\cx[x^{\pm},y^{\pm},z^{\pm}]$ with unknowns $f_1, f_2, f_3, f_4, f_5, f_6 \in \cx[x^{\pm},y^{\pm},z^{\pm}]$ corresponding to the six vertices in $W$:
$$(\widehat{A}+2I)
\begin{pmatrix}
	f_1\\ f_2\\ f_3\\ f_4\\ f_5\\ f_6
\end{pmatrix}
=
\begin{pmatrix}
	2		& 1			& 1			& 1+x	& 1		& 1		\\
	1		& 2			& 1			& 1		& 1+y	& 1		\\
	1+x^{-1}& 1			& 2			& 1		& 1		& 1+z	\\
	1		& 1+y^{-1}	& 1			& 2		& 1		& 1		\\
	1		& 1			& 1			& 1		& 2		& 1		\\
	1		& 1			& 1+z^{-1}	& 1		& 1		& 2		\\
\end{pmatrix}
\begin{pmatrix}
	f_1\\ f_2\\ f_3\\ f_4\\ f_5\\ f_6
\end{pmatrix} = 
\begin{pmatrix}0\\0\\0\\0\\0\\0\end{pmatrix}
.$$
Using row operations over the ring $\cx[x^{\pm},y^{\pm},z^{\pm}]$, the above system is equivalent to:
$$
\begin{pmatrix}
	4		& 0		& 0		& 1+3x	         & 1-y	          & 1-z	\\
	0		& 4		& 0		& 1-x            & 1+3y	          & 1-z	\\
	0       & 0		& 4		& 1-x	         & 1-y	          & 1+3z	\\
	0		& 0     & 0		& (1-z)(1-x^{-1})& (1-y)(1-x^{-1})& (1-z)(1-x^{-1})	\\
	0		& 0		& 0 	& 0		         & 0	          & 0		\\
	0		& 0		& 0	    & 0		         & 0              & 0		\\
\end{pmatrix}
\begin{pmatrix}	f_1\\ f_2\\ f_3\\ f_4\\ f_5\\ f_6\end{pmatrix} = 
\begin{pmatrix}0\\0\\0\\0\\0\\0\end{pmatrix},
$$
and hence our system reduces to the equations:
$$-4f_1 = (1+3x)f_4 + (1-y)f_5 + (1-z)f_6,\;\;\;\;\;\;(1)$$
$$-4f_2 = (1-x)f_4 + (1+3y)f_5 + (1-z)f_6,\;\;\;\;\;\;(2)$$
$$-4f_3 = (1-x)f_4 + (1-y)f_5 + (1+3z)f_6,\;\;\;\;\;\;(3)$$
$$(1-x)f_4 + (1-y)f_5 + (1-z)f_6 = 0.$$
In particular, if we solve the last equation, we can get $f_1, f_2,$ and $f_3$ from equations $(1)-(3)$. The last equation is equivalent to computing the syzygy module of $<1-x,1-y,1-z> \subset \cx[x^{\pm1},y^{\pm1},z^{\pm1}]$ with respect to the generating set $\{1-x,1-y,1-z\}$.
At this point, we need to use some commutative algebra, namely Theorem 3.2 from Chapter 5 in \cite{UsingAlgGeom}. 
This theorem provides generators for $Syz(M)$ of a module $M$ with respect to a generating set which is a Gr{\"o}bner basis.
Fix the monomial order $x<y<z$.
The set $\{1-x,1-y,1-z\}$ has $S$-polynomials:
$$S(1-x,1-y) = y (1-x) - x (1-y) = y - x = (1-x) - (1-y),$$
$$S(1-x,1-z) = z (1-x) - x (1-z) = z - x = (1-x) - (1-z),$$
$$S(1-y,1-z) = z (1-y) - x (1-z) = z - y = (1-y) - (1-z),$$
where the last equality in each line is the result of a long division with respect to $\{1-x,1-y,1-z\}$.
Since we get zero remainders, by Buchberger's criterion, $\{1-x,1-y,1-z\}$ is a Gr{\"o}bner basis.
By Theorem 3.2 from Chapter 5 in \cite{UsingAlgGeom}, $Syz(1-x,1-y,1-z) \subset \bigoplus_{i=1}^3 \cx[x^{\pm1},y^{\pm1},z^{\pm1}]$ is generated by the following three elements:
$$ s_{12} = \begin{pmatrix} -(1-y)\\    (1-x)\\    0        \end{pmatrix},\;
s_{13} = \begin{pmatrix}    -(1-z)\\        0\\   (1-x)     \end{pmatrix},\;
s_{23} = \begin{pmatrix}    0\\         (1-z)\\   -(1-y)     \end{pmatrix}. $$
Therefore, the solution set to equation $(1-x)f_4 + (1-y)f_5 + (1-z)f_6 = 0$ is parametrized by $h_1, h_2, h_3 \in \cx[x^{\pm1},y^{\pm1},z^{\pm1}]$ and is given by
$$
\begin{pmatrix} f_4 \\ f_5 \\ f_6   \end{pmatrix} = h_3 s_{12} + h_2 s_{13} + h_1 s_{23} = 
\begin{pmatrix}
    -(1-y) h_3 - (1-z) h_2 \\
    (1-x) h_3 + (1-z) h_1 \\
    (1-x) h_2 - (1-y) h_1
\end{pmatrix}
$$
Plugging in the solution for $(f_4, f_5, f_6)^T$, we can solve for $(f_1, f_2, f_3)^T$ using equations $(1)-(3)$:
$$
\begin{pmatrix} f_1 \\ f_2 \\ f_3   \end{pmatrix} = 
\begin{pmatrix}
    (x-xy) h_3 - (x-xz) h_2 \\
    (-y+xy) h_3 + (-y+yz) h_1 \\
    (-z+xz) h_2 - (-z+yz) h_1
\end{pmatrix}
$$
Putting it all together, the set of solutions to $(\widehat{A} + 2I)\cdot (f_1, f_2, f_3, f_4, f_5, f_6)^T = \Vec{0}$ is parameterized by $h_1, h_2, h_3 \in \cx[x^{\pm1},y^{\pm1},z^{\pm1}]$ and is given by:
$$
\begin{pmatrix}f_1\\f_2\\f_3 \\f_4\\f_5\\f_6 \end{pmatrix}
= h_1 \begin{pmatrix}
 0\\-y+yz\\z-yz\\0\\1-z\\-1+y   \end{pmatrix}
+ h_2 \begin{pmatrix}
 x-xz\\0\\-z+xz\\-1+z\\0\\1-x \end{pmatrix}  
+ h_3 \begin{pmatrix}
 x-xy\\-y+xy\\0\\-1+y\\1-x\\0 \end{pmatrix}
= h_1 \widehat{f}^{(1)} + h_2 \widehat{f}^{(2)} + h_3 \widehat{f}^{(3)},$$
where $\widehat{f}^{(1)}, \widehat{f}^{(2)}, \widehat{f}^{(3)}$ are the Floquet-Bloch transforms of the eigenfunctions in Figure 5.

It remains to show that $Syz(\widehat{f}^{(1)}, \widehat{f}^{(2)}, \widehat{f}^{(3)}) = <(1-x,1-y,1-z)>$, i.e. to find all solutions $(p_1,p_2,p_3)\in\bigoplus_{i=1}^3 \cx[x^{\pm1},y^{\pm1},z^{\pm1}]$ to the equation $p_1 \widehat{f}^{(1)} + p_2 \widehat{f}^{(2)} + p_3 \widehat{f}^{(3)} = 0$. Equivalently, we wish to solve the system:
$$
\begin{pmatrix}
    0       & x-xz  & -x+xy\\
    -y+yz   &0      & y-xy\\
    z-yz    & -z+xz & 0\\
    0       & -1-z  &1-y\\
    1-z     & 0     &-1+x\\
    -1+y    & 1-x   & 0 \\
\end{pmatrix}
\begin{pmatrix}p_1\\p_2\\p_3\end{pmatrix}=
\begin{pmatrix}0\\0\\0\\0\\0\\0\end{pmatrix}\;\;\;(\star).
$$
Looking at the $4^{th}$ row, we get 
$$-(1-z) p_2 + (1-y) p_3 = 0.$$
Since $\cx[x^{\pm1},y^{\pm1},z^{\pm1}]$ is a unique factorization domain, $(1-z)|p_3$ and hence there is $q_3 \in \cx[x^{\pm1},y^{\pm1},z^{\pm1}]$ such that $p_3 = (1-z)q_3$.
Similarly, we can find $q_2$ such that $p_2 = (1-y) q_2$.
The $4^{th}$ row then becomes $-(1-z)(1-y)q_2+(1-z)(1-y)q_3=0$, hence $q_2=q_3=:q$.
Next, looking at the $5^{th}$ row we have
$$(1-z) p_1 - (1-x)(1-z)q = 0,$$ hence $(1-x)|p_1$, hence we can find $q_1$ such that $p_1 = (1-x)q_1$.
The $5^{th}$ row then becomes $ (1-z)(1-x)q_1 - (1-x)(1-z)q = 0$, hence $q_1 = q$.

We have shown that if $(p_1,p_2,p_3)^T$ solves $(\star)$, then there exists $q \in \cx[x^{\pm1},y^{\pm1},z^{\pm1}]$ such that $(p_1,p_2,p_3)^T = ((1-x)q, (1-y)q, (1-z)q)$.
Since we have already seen that $(1-x,1-y,1-z) \in Syz(\widehat{f}^{(1)}, \widehat{f}^{(2)}, \widehat{f}^{(3)})$, we conclude that
$Syz(\widehat{f}^{(1)}, \widehat{f}^{(2)}, \widehat{f}^{(3)}) = <(1-x,1-y,1-z)>$
and Proposition 7.1 follows.

\section*{Acknowledgements}
The author is grateful to Prof. Rostislav Grigorchuk for invaluable advice, support and guidance throughout the project and to Prof. Peter Kuchment for additional invaluable advice on writing and context. 
The author is also grateful Prof. Christophe Pittet for interest to this work and valuable remarks.
Last but not least, the author is grateful to the referees for pointing out lots of related literature and plenty of small mistakes and typos throughout the paper.

\newpage
\bibliographystyle{plain}
\bibliography{references}

\def\cprime{$'$}
\begin{thebibliography}{10}

\bibitem{BerkKuchQuantum}
Gregory Berkolaiko and Peter Kuchment.
\newblock {\em Introduction to quantum graphs}, volume 186 of {\em Mathematical
  Surveys and Monographs}.
\newblock American Mathematical Society, Providence, RI, 2013.

\bibitem{UsingAlgGeom}
David~A. Cox, John Little, and Donal O'Shea.
\newblock {\em Using algebraic geometry}, volume 185 of {\em Graduate Texts in
  Mathematics}.
\newblock Springer, New York, second edition, 2005.

\bibitem{DelyonSouillardOriginalIdea}
Fran\c{c}ois Delyon and Bernard Souillard.
\newblock Remark on the continuity of the density of states of ergodic finite
  difference operators.
\newblock {\em Comm. Math. Phys.}, 94(2):289--291, 1984.

\bibitem{DoKuchSott}
Ngoc Do, Peter Kuchment, and Frank Sottile.
\newblock Generic properties of dispersion relations for discrete periodic
  operators.
\newblock {\em J. Math. Phys.}, 61(10):103502, 19, 2020.

\bibitem{StrongAtiyahAndAlgebra}
J\'{o}zef Dodziuk, Peter Linnell, Varghese Mathai, Thomas Schick, and Stuart
  Yates.
\newblock Approximating {$L^2$}-invariants and the {A}tiyah conjecture.
\newblock volume~56, pages 839--873. 2003.
\newblock Dedicated to the memory of J\"{u}rgen K. Moser.

\bibitem{GGTDrutuKapovich}
Cornelia Dru\c{t}u and Michael Kapovich.
\newblock {\em Geometric group theory}, volume~63 of {\em American Mathematical
  Society Colloquium Publications}.
\newblock American Mathematical Society, Providence, RI, 2018.
\newblock With an appendix by Bogdan Nica.

\bibitem{AnalysisOnGraphsSurvey}
Pavel Exner, Jonathan~P. Keating, Peter Kuchment, Toshikazu Sunada, and
  Alexander Teplyaev, editors.
\newblock {\em Analysis on graphs and its applications}, volume~77 of {\em
  Proceedings of Symposia in Pure Mathematics}.
\newblock American Mathematical Society, Providence, RI, 2008.
\newblock Papers from the program held in Cambridge, January 8--June 29, 2007.

\bibitem{FaustSottile}
Matthew Faust and Frank Sottile.
\newblock Critical points of discrete periodic operators, 2022.
\newblock in preparation.

\bibitem{FLM}
Jake Fillman, Wencai Liu, and Rodrigo Matos.
\newblock Irreducibility of the {B}loch variety for finite-range
  {S}chr\"odinger operators, 2021.
\newblock {\tt arXiv:2107.06447}.

\bibitem{GKT}
D.~Gieseker, H.~Kn\"{o}rrer, and E.~Trubowitz.
\newblock {\em The geometry of algebraic {F}ermi curves}, volume~14 of {\em
  Perspectives in Mathematics}.
\newblock Academic Press, Inc., Boston, MA, 1993.

\bibitem{GriPittetNoEigenv}
Rostislav Grigorchuk and Christophe Pittet.
\newblock Laplace and schrödinger operators without eigenvalues on homogeneous
  amenable graphs.
\newblock 2021.
\newblock preprint.

\bibitem{GriZukLamplighter}
Rostislav~I. Grigorchuk and Andrzej \.{Z}uk.
\newblock The lamplighter group as a group generated by a 2-state automaton,
  and its spectrum.
\newblock {\em Geom. Dedicata}, 87(1-3):209--244, 2001.

\bibitem{PHallPolycyclic}
P.~Hall.
\newblock Finiteness conditions for soluble groups.
\newblock {\em Proc. London Math. Soc. (3)}, 4:419--436, 1954.

\bibitem{Higuchi:Nomura:Paper}
Yusuke Higuchi and Yuji Nomura.
\newblock Spectral structure of the {L}aplacian on a covering graph.
\newblock {\em European J. Combin.}, 30(2):570--585, 2009.

\bibitem{KestenClassic}
Harry Kesten.
\newblock Symmetric random walks on groups.
\newblock {\em Trans. Amer. Math. Soc.}, 92:336--354, 1959.

\bibitem{IDS_for_PDE_Survey}
Werner Kirsch and Bernd Metzger.
\newblock The integrated density of states for random {S}chr\"{o}dinger
  operators.
\newblock In {\em Spectral theory and mathematical physics: a {F}estschrift in
  honor of {B}arry {S}imon's 60th birthday}, volume~76 of {\em Proc. Sympos.
  Pure Math.}, pages 649--696. Amer. Math. Soc., Providence, RI, 2007.

\bibitem{KittelCondensedMatter}
Charles Kittel.
\newblock {\em Introduction to Solid State Physics}.
\newblock Wiley, 6th edition.

\bibitem{Kuch}
P.~A. Kuchment.
\newblock On the {F}loquet theory of periodic difference equations.
\newblock In {\em Geometrical and algebraical aspects in several complex
  variables ({C}etraro, 1989)}, volume~8 of {\em Sem. Conf.}, pages 201--209.
  EditEl, Rende, 1991.

\bibitem{KagomeAndFiniteSupport}
Daniel Lenz, Norbert Peyerimhoff, Olaf Post, and Ivan Veseli\'{c}.
\newblock Continuity of the integrated density of states on random length
  metric graphs.
\newblock {\em Math. Phys. Anal. Geom.}, 12(3):219--254, 2009.

\bibitem{LenzVeselicUniformApproximationAndExamples}
Daniel Lenz and Ivan Veseli\'{c}.
\newblock Hamiltonians on discrete structures: jumps of the integrated density
  of states and uniform convergence.
\newblock {\em Math. Z.}, 263(4):813--835, 2009.

\bibitem{Liu20}
Wencai Liu.
\newblock Irreducibility of the {F}ermi variety for discrete periodic
  {S}chr\"odinger operators and embedded eigenvalues.
\newblock {\tt arXiv:2006.04733}, 2020.

\bibitem{GGTLoeh}
Clara L\"{o}h.
\newblock {\em Geometric group theory}.
\newblock Universitext. Springer, Cham, 2017.
\newblock An introduction.

\bibitem{LuckLTwoInvariants}
Wolfgang L\"{u}ck.
\newblock {\em {$L^2$}-invariants: theory and applications to geometry and
  {$K$}-theory}, volume~44 of {\em Ergebnisse der Mathematik und ihrer
  Grenzgebiete. 3. Folge. A Series of Modern Surveys in Mathematics [Results in
  Mathematics and Related Areas. 3rd Series. A Series of Modern Surveys in
  Mathematics]}.
\newblock Springer-Verlag, Berlin, 2002.

\bibitem{survey}
Bojan Mohar and Wolfgang Woess.
\newblock A survey on spectra of infinite graphs.
\newblock {\em Bull. London Math. Soc.}, 21(3):209--234, 1989.

\bibitem{Pastur}
L.~A. Pastur.
\newblock Spectral properties of disordered systems in the one-body
  approximation.
\newblock {\em Comm. Math. Phys.}, 75(2):179--196, 1980.

\bibitem{PasturFigotin}
Leonid Pastur and Alexander Figotin.
\newblock {\em Spectra of random and almost-periodic operators}, volume 297 of
  {\em Grundlehren der mathematischen Wissenschaften [Fundamental Principles of
  Mathematical Sciences]}.
\newblock Springer-Verlag, Berlin, 1992.

\bibitem{ReedSimonClassic}
Michael Reed and Barry Simon.
\newblock {\em Methods of modern mathematical physics. {I}}.
\newblock Academic Press, Inc. [Harcourt Brace Jovanovich, Publishers], New
  York, second edition, 1980.
\newblock Functional analysis.

\bibitem{ShubinMagnetic}
M.~A. Shubin.
\newblock Discrete magnetic {L}aplacian.
\newblock {\em Comm. Math. Phys.}, 164(2):259--275, 1994.

\bibitem{SunadaMagnetic}
Toshikazu Sunada.
\newblock A discrete analogue of periodic magnetic {S}chr\"{o}dinger operators.
\newblock In {\em Geometry of the spectrum ({S}eattle, {WA}, 1993)}, volume 173
  of {\em Contemp. Math.}, pages 283--299. Amer. Math. Soc., Providence, RI,
  1994.

\bibitem{VeselicPercolationHamiltonians}
Ivan Veseli\'{c}.
\newblock Spectral analysis of percolation {H}amiltonians.
\newblock {\em Math. Ann.}, 331(4):841--865, 2005.

\bibitem{Shubin}
M.~A. \v{S}ubin.
\newblock Spectral theory and the index of elliptic operators with
  almost-periodic coefficients.
\newblock {\em Uspekhi Mat. Nauk}, 34(2(206)):95--135, 1979.

\end{thebibliography}

\end{document}